\documentclass[letterpaper,11pt]{article}
\usepackage{color}
\usepackage{pdflscape}

\usepackage[small]{titlesec}
\usepackage{theorem,amsmath,amssymb,amscd}
\usepackage[all,cmtip]{xy}
\usepackage{booktabs}
\usepackage[hyperfootnotes=false, colorlinks, linkcolor={blue}, citecolor={magenta}, filecolor={blue}, urlcolor={blue}]{hyperref}
\usepackage[overload]{textcase} 

\usepackage{hyperref}
\theoremstyle{change}
\allowdisplaybreaks
\newcommand{\Q}{{\mathbb Q}}
\newcommand{\Z}{{\mathbb Z}}

\newcommand{\p}{\mathfrak p}
\newcommand{\OF}{{\mathfrak o}}
\newcommand{\Fq}{{\mathbb F}_q}
\newcommand{\GL}{{\rm GL}}

\newcommand{\SL}{{\rm SL}}

\newcommand{\GSp}{{\rm GSp}}

\newcommand{\Sp}{{\rm Sp}}

\newcommand{\K}[1]{{\rm K}(\p^{#1})}
\newcommand{\Kl}[1]{{\rm Kl}(\p^{#1})}

\newcommand{\val}{{\rm val}}

\newcommand{\cInd}{\text{\rm c-Ind}}

\newcommand{\trace}{{\rm tr\,}}

\newcommand{\Supp}{{\rm Supp}}

\newcommand{\qed}{\hspace*{\fill}\rule{1ex}{1ex}}
\newcommand{\forget}[1]{}
\def\qdots{\mathinner{\mkern1mu\raise0pt\vbox{\kern7pt\hbox{.}}\mkern2mu
\raise3.4pt\hbox{.}\mkern2mu\raise7pt\hbox{.}\mkern1mu}}

\newenvironment{proof}{\vspace{1ex}\noindent\emph{Proof.}\hspace{0.5em}}
	{\hfill\qed\vspace{2ex}}
\newenvironment{bsmallmatrix}{\left[\begin{smallmatrix}}{\end{smallmatrix}\right]}

\newtheorem{lemma}{Lemma.}[section]
\newtheorem{theorem}[lemma]{Theorem.}
\newtheorem{corollary}[lemma]{Corollary.}
\newtheorem{proposition}[lemma]{Proposition.}

\newtheorem{remark}[lemma]{Remark.}

\begin{document}

\thispagestyle{empty}
\begin{center}
 {\bf\Large Klingen Vectors for Depth Zero Supercuspidals of $\GSp(4)$}

 \vspace{3ex}
 Jonathan Cohen
 
 \vspace{3ex}
 \begin{minipage}{80ex}
  \small\textbf{Abstract.} Let $F$ be a non-archimedean local field of characteristic zero and $(\pi, V)$ a depth zero, irreducible, supercuspidal representation of $\GSp(4, F)$. We calculate the dimensions of the spaces of Klingen-invariant vectors in $V$ of level $\p^n$ for all $n\geq0$. 
 \end{minipage}
 \vspace{3ex}
\end{center}

\section{Introduction} This paper is concerned with a dimension counting problem in $p$-adic representation theory for the group $\GSp(4)$. Part of the motivation comes from the following problem in the classical theory of Siegel modular forms. If $\Gamma\subset \Sp(4, \Q)$ is a congruence subgroup, then the dimension of the space of cusp forms of level $\Gamma$ (and integer weight $k$, say) is not known in general. For example, if $$\Gamma=\begin{bmatrix}
\Z & 4\Z & \Z & \Z \\
\Z & \Z & \Z & \Z \\
\Z & 4\Z & \Z & \Z \\
4\Z & 4\Z & 4\Z & \Z 
\end{bmatrix}\cap \Sp(4, \Z)$$ is the ``Klingen congruence subgroup of level $4$'' then the associated dimensions were only recently computed in \cite{RSY2022}. The method by which this was achieved required, as one of its several inputs, the  dimensions of spaces of fixed vectors in all irreducible smooth representations of $\GSp(4, \Q_2)$ for the subgroups $${\rm Kl}(4)=\begin{bmatrix}
\Z_2 & \Z_2 & \Z_2 & \Z_2 \\
4\Z_2 & \Z_2 & \Z_2 & \Z_2 \\
4\Z_2 & \Z_2 & \Z_2 & \Z_2 \\
4\Z_2 & 4\Z_2 & 4\Z_2 & \Z 
\end{bmatrix}\cap \GSp(4, \Z_2).$$ These dimensions, and more, had been computed in \cite{Yi2021}. We remark that the different ``shape'' of the subgroups in question is an artifact of different conventions for alternating forms in the classical and representation-theoretic contexts. 

If one hopes to use the approach taken in \cite{RSY2022} for other congruence subgroups, then a necessary component is the determination of dimensions of spaces of fixed vectors in all irreducible smooth representations of $\GSp(4, \Q_p)$ for appropriate local subgroups. In this paper we are concerned with the subgroups $${\rm Kl}(p^n)=\begin{bmatrix}
\Z_p & \Z_p & \Z_p &  \Z_p \\
p^n\Z_p & \Z_p & \Z_p &  \Z_p \\
p^n\Z_p & \Z_p & \Z_p &  \Z_p \\
p^n\Z_p & p^n\Z_p & p^n\Z_p &  \Z_p \\
\end{bmatrix}\cap \GSp(4, \Z_p).$$ At this time, computing the associated dimensions for all representations appears overly ambitious, so we restrict ourselves to the very special case of depth zero supercuspidals (constructed in the next section). While this is a limitation, some interesting phenomena already arise that had not been observed for $n\leq 2$. The most significant of these is that the dimensions now depend on $p$ (once $n\geq 4$), and on finer internal structure of these supercuspidals (once $n\geq3$). This internal structure can also be expressed in terms of Langlands parameters or of $L$-packets; see Remark \ref{L packet remark}. We note in particular that dependence on $p$ never occurs when one instead considers the sequence paramodular subgroups or stable Klingen subgroups; see \cite{Ro2007} and \cite{J2022}. This difference has implications for how one might hope a local newform and oldform theory would work for the Klingen congruence subgroups. For example, it suggests that one would need a plethora of raising operators to obtain all oldforms.

We give a brief outline of the contents of this paper. After setting up notation and definitions in section 2, we show in section 3 that one class of depth zero supercuspidal representations (all of which are nongeneric) will not have any fixed vectors for $\Kl{n}$; see Theorem \ref{Nongeneric Main Theorem}. In section~4 we consider the remaining depth zero supercuspidals, and reduce to the determination and enumeration of a certain set of double cosets, together with a collection of character table computations for $\GSp(4, \Fq)$. Along the way, we can conclude that depth zero supercuspidals will have no fixed vectors for $\Kl{n}$ (for all $n\geq 0$) if they are nongeneric. We conjecture that this remains true for all nongeneric supercuspidals. The main result, our dimension formula, is Theorem~\ref{Main Theorem}.

{\bf Acknowledgments.} This work was done partially while the author was participating in the virtual program of the Institute for Mathematical Sciences, National University of Singapore, in 2022. We also thank Ralf Schmidt for many useful conversations during the writing of this article.

\section{Setup}

Let $F$ be a finite extension of $\Q_p$, with ring of integers $\OF$, maximal ideal $\p$, uniformizer $\varpi$, and residue field $\Fq$. Let $J = \begin{bsmallmatrix}
& & & 1\\
& & 1 & \\
 &-1 & & \\
 -1 & & & 
\end{bsmallmatrix}$ and let $G=\GSp(4, F)=\{ g\in \GL(4, F): {^t}gJg=\mu(g)J  \}$, $Z=Z(G)$, $T\subset G$ the set of diagonal matrices, and $G^0=\{g\in G: \mu(g)\in \OF^\times  \}$. We will need the two maximal compact subgroups $K=\GSp(4, \OF)$ and $$\K{}=\begin{bmatrix}
\OF & \OF & \OF &  \p^{-1} \\
\p & \OF & \OF &  \OF \\
\p & \OF & \OF &  \OF \\
\p & \p & \p &  \OF 
\end{bmatrix}\cap G^0$$ as well as the standard Iwahori subgroup $$I=\begin{bmatrix}
\OF^\times & \OF & \OF &  \OF \\
\p & \OF^\times & \OF &  \OF \\
\p & \p & \OF^\times &  \OF \\
\p & \p & \p &  \OF^\times 
\end{bmatrix}\cap G$$ and the sequence of Klingen congruence subgroups $$\Kl{n}=\begin{bmatrix}
\OF & \OF & \OF &  \OF \\
\p^n & \OF & \OF &  \OF \\
\p^n & \OF & \OF &  \OF \\
\p^n & \p^n & \p^n &  \OF 
\end{bmatrix}\cap G^0$$ where $n$ is a positive integer. Let $K^+$ and $\K{}^+$ be the prounipotent radicals of $K$ and $\K{}^+$, respectively. Define the two matrices 

$$s_1 = \begin{bmatrix}
& 1 & &  \\
1 &  &  & \\
& & & 1 \\
& & 1 & 
\end{bmatrix}  , \qquad s_2 = \begin{bmatrix}
1&  & &  \\
  &  & 1  & \\
&  -1 & &  \\
& &  & 1 
\end{bmatrix}.$$
 The images of $s_1$ and $s_2$ generate the Weyl group $W:=N_G(T)/T$. 

If $H$ is a subgroup of $G$ and $\rho:H\to \GL(W)$ is a representation of $H$, then $\cInd_H^G(\rho)$ is the space of functions $f:G\to W$ such that $f(hg)=\rho(h)f(g)$ for all $h\in H$, $g\in G$, with $f$ being smooth and having compact support modulo $H$. This space is a smooth representation of $G$ under right translation. 

If $\pi$ is a depth zero supercuspidal representation of $G$ then $\pi$ arises in one of two ways. Either $\pi = \cInd_{ZK}^G(\sigma)$ where $\sigma|_K$ is an inflation of an irreducible cuspidal representation of $K/K^+$, or  $\pi = \cInd_{N_G(\K{})}^G (\tau)$ where $\tau$ is an irreducible representation of $N_G(\K{})$ and $\tau|_{\K{}}$ is an inflation of a cuspidal representation of $\K{}/\K{}^+$. We will compute $\dim \pi^{\Kl{n}}$ for all such~$\pi$. Since $\pi$ is supercuspidal, it has no fixed vectors for an Iwahori subgroup, and in particular $\pi^{\Kl{}}=0$. We therefore need consider only fixed vectors for $\Kl{n}$ with $n\geq 2$.

\section{Klingen vectors for \texorpdfstring{$\pi = \cInd_{N_G(\K{})}^G (\tau)$}{}}

We will first consider those depth zero supercuspidals coming from $\K{}$. It turns out that these do not admit any Klingen vectors of any level. 

\begin{theorem} \label{Nongeneric Main Theorem}
Let $\pi = \cInd_{N_G(\K{})}^G (\tau)$ as above. Then $\pi^{\Kl{n}}=0$ for all $n\geq 0$. 
\end{theorem}

\begin{proof} For $g\in G$ and $n\geq 1$ define $$R_g := \left(g\Kl{n}g^{-1}\cap \K{}\right)\K{}^+ =\left( g\Kl{n}g^{-1}\cap N_G(\K{})\right)\K{}^+ .$$ If $f\in \pi^{\Kl{n}}$ then $f(g) \in \tau^{R_g}$. It suffices to show $\tau^{R_g}=0$ for $g$ ranging over a set of representatives for $N_G(\K{})\backslash G/\Kl{n}$. We have $G = I N_G(T) I$ by a well-known decomposition (see \cite{Iw1965}). Since $N_G(\K{})$ contains $I$ along with a set of representatives for $N_G(T)/T$, and using the Iwahori factorization of $I$, we may take 
$$g=\begin{bmatrix}
\varpi^{2i+j} & & & \\
 & \varpi^{i+j} & & \\
 & & \varpi^i  & \\
 & & & 1
\end{bmatrix} \begin{bmatrix}
1 &  & & \\
x & 1 & & \\
y & & 1 & \\
z & y & -x & 1
\end{bmatrix}$$ where $i, j\in \Z$, $x, y, z\in \p$. 
Using $s_2\in \Kl{n}\cap \K{}$ we may assume that $i+j\geq i$, so $j\geq 0$. 

If $x = 0$ then let $d\in \OF$ and compute $$g\begin{bmatrix}
1 & & & \\ 
& 1 & & \\
& \varpi^jd & 1 & \\
& & & 1
\end{bmatrix}g^{-1} = \begin{bmatrix}
1 & & & \\ 
& 1 & & \\
& d & 1 & \\
& & & 1
\end{bmatrix}.$$ Thus $R_g\supset \begin{bsmallmatrix}
1 & & & \\
& 1 & & \\
 & \OF & 1 & \\
  & & & 1
\end{bsmallmatrix}.
$ The image of $\begin{bsmallmatrix}
1 & & & \\
& 1 & & \\
 & \OF & 1 & \\
  & & & 1
\end{bsmallmatrix}
$ in $\K{}/\K{}^+$ is the unipotent radical of a parabolic subgroup. So $\tau^{R_g}=0$ since $\tau|_{\K{}}$ is cuspidal. 

If $x\neq 0$ then we may assume $1\leq \val(x) \leq i$, since $g = S(\varpi^{-i}x, 0, 0) t_{i,j} S(0, y, z+xy)$, and $S(\varpi^{-i}x, 0, 0)\in \K{}$ if $\val(x)>i$. Then let $d\in \p^{2i+j+1-2\val(x)} \subset \p^{j+1}$, so $dx \in  \p^{i+j+1}$, and compute $$g\begin{bmatrix}
1 & & & \\ 
& 1 & & \\
& d & 1 & \\
& & & 1
\end{bmatrix}g^{-1} = \begin{bmatrix}
1 & & & \\ 
& 1 & & \\ 
-\varpi^{-i-j}dx & \varpi^{-j}d & 1 & \\
-\varpi^{-2i-j}dx^2 & -\varpi^{-i-j}dx & & 1
\end{bmatrix}.$$ Thus $R_g \supset \begin{bsmallmatrix}
1 & & & \\
& 1 & & \\
 &  & 1 & \\
\p  & & & 1
\end{bsmallmatrix}.$ The image of $\begin{bsmallmatrix}
1 & & & \\
& 1 & & \\
 &  & 1 & \\
\p  & & & 1
\end{bsmallmatrix}
$ in $\K{}/\K{}^+$ is the unipotent radical of a parabolic subgroup. So $\tau^{R_g}=0$ since $\tau|_{\K{}}$ is cuspidal. \end{proof}

\section{Klingen vectors for \texorpdfstring{$\pi = \cInd_{ZK}^G(\sigma)$}{}}

For $g\in G$ and fixed $n\geq 1$ define the objects \begin{eqnarray}
[g] &=& ZK g \Kl{n}\nonumber \\
R_g &=& (g\Kl{n}g^{-1}\cap K)K^+/K^+ \nonumber\\
\Supp(\pi) &=& \{ [g] : \sigma^{R_g}\neq 0  \} = \{ [g]: \exists f \in \pi^{\Kl{n}}, \ f(g)\neq 0   \}.  \nonumber
\end{eqnarray}

Then $$\dim \pi^{\Kl{n}} = \sum\limits_{[g]\in \Supp(\pi)} \dim \sigma^{R_g}.$$ We therefore require a determination of $\Supp(\pi)$ and of $\dim \sigma^{R_g}$ for $[g]\in \Supp(\pi)$. 

\subsection{Determination of \texorpdfstring{$\Supp(\pi)$}{}} In this section we obtain a parametrization of  $\Supp(\pi)$; in a subsequent section we will compute the values of $\dim \sigma^{R_g}$ via the character table of $K/K^+ = \GSp(4, \Fq)$. We have $G = I N_G(T) I$ by a well-known decomposition; see \cite{Iw1965}. Since $K\supset I$ and $K\supset \langle s_1, s_2\rangle$, using the Iwahori factorization for $I$ we may take double coset representatives for $ZK \backslash  G/\Kl{n}$ to have the form $$t_{i, j}S(x, y,z) : = \begin{bmatrix}
\varpi^{2i+j} & & & \\
 & \varpi^{i+j} & & \\
 & & \varpi^i  & \\
 & & & 1
\end{bmatrix} \begin{bmatrix}
1 &  & & \\
x & 1 & & \\
y & & 1 & \\
z & y & -x & 1
\end{bmatrix}$$ where $x, y, z\in \p$. Using $s_2\in \Kl{n}\cap K$ we may assume that $i+j\geq i$, so $j\geq 0$. 

\begin{lemma} \label{double coset constraints}
Consider $[t_{i,j}S(x, y, z)]$ with $j\geq 0$. 

a) If $x\neq 0$ we may assume $1\leq \val(x)\leq i-1$. If $y\neq 0$ we may assume $1\leq \val(y)\leq i+j-1$. If $z\neq 0$ we may assume $1\leq \val(z) \leq \min(2i+j, n)-1$. 

b) If $xy\neq 0$ we may assume $1\leq \val(y/x) \leq j-1$. 

c) If $xz\neq 0$, we may assume $1\leq \val(z/x)\leq i+j-1$. 

d) If $yz\neq 0$ we may assume $\min(1, 1+i)\leq \val(z/y) \leq \max( i, \val(y))-1$. 
\end{lemma}

\begin{proof}
a) These conditions follow from the equalities \begin{eqnarray}
t_{i,j}S(x, y, z) &=& S(\varpi^{-i}x,0,0) t_{i,j}S(0,y,z+xy) \nonumber \\ 
&=& S(0,\varpi^{-i-j}y,0)t_{i,j}S(x, 0, z-xy) \nonumber \\
&=& S(0,0,\varpi^{-2i-j}z) t_{i,j}S(x,y,0)  \nonumber \\
&=& t_{i,j}S(x,y,0)S(0,0,z).  \nonumber 
\end{eqnarray}

b) Let $c\in \p^j$ and $A = \begin{bsmallmatrix}
1 & & & \\
& 1 & b  & \\
& c & 1 & \\
& & & 1
\end{bsmallmatrix}\in K$. Then \begin{eqnarray}
 [t_{i, j} S(x, y, z)] &=& \left[\left(t_{i, j}At_{i,j}^{-1}\right)t_{i,j}A^{-1} S(x, y, z) A\right] \nonumber \\
 &=& [t_{i,j}S(x-by, y-cx, z)]. \nonumber
\end{eqnarray} If $\val(x)\geq \val(y)$ then we may take $c=0$ and $b=x/y$ to eliminate $x$. If $\val(y)-\val(x)\geq j$ then we may take $b=0$ and $c=y/x$ to eliminate $y$.   

c) If $z/x\in \p^{i+j}$ then $z/x^2\in \p^{j+1}$ since $x\neq 0$ and using the first part. For $d\in \OF$ let $A(d) = \begin{bsmallmatrix}
1 & & & \\
& 1 & & \\
& d & 1 & \\
& & & 1
\end{bsmallmatrix}$. Then \begin{eqnarray}
[t_{i,j}S(x, y, z)] &=&[t_{i,j}S(x, y, z)A(-z/x^2) ] \nonumber \\ 
&=& [A(-\varpi^{-j}z/x^2) t_{i,j}S(x, y+z/x, z)] \nonumber \\
&=&[t_{i,j} S(0,z/x, 0)S(x,y,0)] \nonumber \\
&=&[S(0,\varpi^{-i-j}z/x, 0)t_{i,j} S(x,y,0)] \nonumber \\
&=& [t_{i,j}S(x,y,0)].  \nonumber 
\end{eqnarray}

If $\val(z/x)\leq 0$ then we may assume $1\leq \val(x)<i$ and $i+j>0$. For $d\in \OF$, let $C(d) = \begin{bsmallmatrix}
1 & & d & \\
& 1 & & d \\
& & 1 & \\
& & & 1
\end{bsmallmatrix}$. Then \begin{eqnarray}
[t_{i,j}S(x, y, z)] &=&[t_{i,j} C(x/z)C(-x/z)S(x,y,z) ] \nonumber \\ 
&=& [C(\varpi^{i+j}x/z) t_{i,j}C(-x/z)S(x,y,z)]  \nonumber \\
&=&[t_{i,j}C(-x/z)S(x,y,z)] \nonumber \\
&=&[t_{i,j}S(0,y', z') h] \nonumber    \\
&=& [t_{i,j}S(0,y, z)]. \nonumber
\end{eqnarray} where $y' = y/(1-xy/z)$, $z' = z/(1-xy/z)$, $$h = \begin{bmatrix}
1-xy/z & & -x/z & \\
& 1-xy/z & x^2/z & -x/z \\
& & 1+xy'/z & \\
& & & 1+ xy'/z
\end{bmatrix}\in \Kl{n}$$ and we used integral diagonal matrices for the last equality.  

d) Suppose $\val(z/y) \geq \max(i, \val(y))$. For $d\in \OF$ let $B(d) = \begin{bsmallmatrix}
1 & & & \\
& 1 & d & \\
&  & 1 & \\
& & & 1
\end{bsmallmatrix}$. Then \begin{eqnarray}
[t_{i,j}S(x, y, z)] &=&[t_{i,j}S(x,y,z) B(z/y^2)] \nonumber \\ 
&=& [B(\varpi^j z/y^2) t_{i,j}S(x-z/y,y,z)]  \nonumber \\
&=& [t_{i,j} S(-z/y, 0,0) S(x,y,0)] \nonumber \\
&=& [S(-\varpi^{-i}z/y, 0,0)t_{i,j} S(x,y,0)]  \nonumber \\
&=& [t_{i,j}S(x,y,0)]. \nonumber
\end{eqnarray}

Finally, suppose $\val(z/y)\leq \min(0, i)$. Let $D(d) = \begin{bsmallmatrix}
1 & d &  & \\
& 1 & & \\
 & & 1 & -d \\
 & & & 1
\end{bsmallmatrix}$. Then \begin{eqnarray}
[t_{i,j}S(x, y, z)] &=&[t_{i,j} D(-y/z)D(y/z)S(x,y,z) ] \nonumber \\ 
&=& [D(-\varpi^{i}y/z) t_{i,j}D(y/z)S(x,y,z)]  \nonumber \\
&=&[t_{i,j}D(y/z)S(x,y,z)] \nonumber \\
&=&[t_{i,j}S(x',0, z') h] \nonumber    \\
&=& [t_{i,j}S(x,0, z)]. \nonumber
\end{eqnarray} where $x' = x/(1+xy/z)$, $z' = z/(1+xy/z)$, $$h = \begin{bmatrix}
1+xy/z & y/z & & \\
& 1-x'y/z &  &  \\
& & 1+xy/z & -y/z \\
& & & 1- x'y/z
\end{bmatrix}\in \Kl{n}$$ and we used integral diagonal matrices for the last equality.  
\end{proof}

Define the two unipotent radicals $U_S,  U_K \subset \GSp(4, \Fq)$ by $$U_S =\left\{\begin{bmatrix}
1 & & & \\
& 1 & & \\
x & y & 1 & \\
z & x &  & 1
\end{bmatrix}: x,y,z\in \Fq\right\} , \qquad U_K = \left\{ \begin{bmatrix}
1 & & & \\
x & 1 & & \\
y & & 1 & \\
z & y & -x  & 1
\end{bmatrix}: x,y,z\in \Fq \right\}.$$ 
Since $\sigma$ is cuspidal, $\sigma^{R_g}=0$ if $R_g$ contains a conjugate of $U_S$ or $U_K$.

\begin{lemma} \label{Main Support Restrictions}
Let $j\geq 0$ and $[t_{i,j}S(x,y,z)]\in \Supp(\pi)$. Then $2-n\leq i\leq n-1$ and $2i+j\geq 1$. 
\end{lemma}

\begin{proof}
Let 
$g = t_{i,j}S(x,y,z)$. We compute $$g\begin{bmatrix}
1 &  &  &  \\
\varpi^n c_1 & 1 &  &  \\
\varpi^n c_2 &  & 1 &  \\
\varpi^n c_3 & \varpi^n c_2 & -\varpi^n c_1 & 1
\end{bmatrix} g^{-1} = \begin{bmatrix}
1 & & & \\
\varpi^{n-i} c_1 & 1 & & \\
\varpi^{n-i-j} c_2 & & 1 & \\
\varpi^{n-2i-j}(c_3 - 2c_2 x + 2c_1 y) & \varpi^n c_2 & -\varpi^n c_1 & 1
\end{bmatrix}.$$ If $i\geq n$ then $i+j\geq n$ and $2i+j\geq n$. Taking $c_1\in \p^{i-n}$, $c_2\in \p^{j-j-n}$, and $c_3\in 2c_2 x - 2c_1 y + \p^{2i+j-n}$ shows $R_g\supset U_K$. Thus $i\leq n-1$. If $2i+j\leq 0$ then $i\leq 0$ and we may assume $x=z=0$. We compute $$g\begin{bmatrix}
1 & b_1 &  & b_3 \\
 & 1 &  &  \\
 & d  & 1 & -b_1 \\
 &  &  & 1
\end{bmatrix}g^{-1} = \begin{bmatrix}
1 & \varpi^i (b_1-yb_3) &  & \varpi^{2i+j}b_3 \\
 & 1 &  &  \\
 & \varpi^{-j} (d+y(2b_1 - yb_3) & 1 & -\varpi^{i}(b_1-yb_3)\\
 &  &  & 1
\end{bmatrix}.$$ Taking $b_1 \in yb_3 + \p^{-i}$, $b_3\in \p^{-2i-j}$, and $d \in -y(2b_1 - yb_3) + \p^{j}$ shows that $R_g$ contains a conjugate of $U_S$. Thus $2i+j\geq 1$. Finally, if $i\leq 1-n\leq -1$ then we may assume $x=0$. Note $i+j \geq -i+1 \geq n$. If $z\neq 0$, then we compute
$$g\begin{bmatrix}
a &  &  &  \\
 & 1 &  &  \\
\varpi^n c & d  & a &  \\
  & \varpi^n c a^{-1} &  & 1
\end{bmatrix}g^{-1} 
= \begin{bmatrix}
a &  &  & \\
 & 1 &  &  \\
\varpi^{n-i-j} c & \varpi^{-j} d & a & \\
\varpi^{-2i-j}z(a -1) & \varpi^{n-i-j} ca^{-1}  &  & 1
\end{bmatrix}.$$ We may assume $\val(z) < 2i+j$, and let $d \in \p^j$, $c\in \p^{i+j-n}$, and $a\in 1+\p^{2i+j-\val(z)}$. Then $R_g\supset U_S$. If $z= 0$, then we compute $$g\begin{bmatrix}
1 & b & & \\
& 1 & & \\
\varpi^n c & d & 1 & -b \\
 & \varpi^n c & & 1
\end{bmatrix}g^{-1} = \begin{bmatrix}
1 & \varpi^{i}b & & \\
& 1 & & \\
\varpi^{n-i-j} c & \varpi^{-j}(d + 2yb) & 1 & -\varpi^i b \\
 & \varpi^{n-i-j} c & & 1
\end{bmatrix}$$ so taking $b\in \p^{-i}$, $c\in \p^{i+j-n}$, and $d \in -2yb + \p^j$ shows $R_g$ contains a conjugate of $U_K$. Thus $2-n\leq i$. 
\end{proof}

\begin{proposition} \label{boundary cases imply diagonal}
Suppose $j\geq 0$ and $[g]:=[t_{i,j}S(x,y,z)]\in \Supp(\pi)$. 

a) If $j=0$ then $[g]=[t_{i,j}]$. 

b) If $i\leq 0$ then  $[g]=[t_{i,j}]$.

\end{proposition}

\begin{proof}
a) We may assume $xy=0$ and $1\leq i\leq n-1$. Conjugating with $s_2$ if necessary, we may assume $x=0$ and compute $$g\begin{bmatrix}
1 & & & \\
& d_1 & d_2 & \\
& d_3 & d_4 & \\
& & & \Delta 
\end{bmatrix}g^{-1} = \begin{bmatrix}
1 & & & \\
-\varpi^{-i}yd_2 & d_1 & d_2 & \\
\varpi^{-i} y(1-d_4)  & d_3 & d_4 & \\
\varpi^{-2i}(z(1-\Delta)  -y^2d_2 ) & \varpi^{-i} y(d_1-\Delta ) & \varpi^{-i}yd_2 & \Delta 
\end{bmatrix}.$$ where $\Delta = d_1d_4 - d_2d_3$. If $1\leq \val(y)\leq i-1$ then taking $d_4 = d_1^{-1}(1+d_2d_3)$, $d_3\in \OF$, $d_2\in \p^{2i-\val(y^2)} \subset \p^{i-\val(y)+1}\subset \p$, and $d_1\in 1+\p^{i-\val(y)}$, shows $R_g\supset U_S$. So we may assume $y=0$. If $1\leq \val(z)< 2i$ then taking $d_i$ so that $\Delta\in 1+\p^{2i-\val(z)}$ shows $R_g\supset \begin{bsmallmatrix}
1 & &  \\
& \SL_2(q) & \\
*  & & 1
\end{bsmallmatrix}$. Then Lemma \ref{Klingen Levi}  shows $\sigma^{R_g}=0$. 
 
 b) We may assume $x=0$. Note $j\geq i+j\geq 2i+j\geq 1$. Then $$g\begin{bmatrix}
 d_1d_4 & d_1b & & \\
 & d_1 &  & \\
 & d_3 & d_4 & -b\\
 & & & 1
 \end{bmatrix}g^{-1} = \begin{bmatrix}
 d_1d_4 & \varpi^i d_1b & & \\
 & d_1 &  & \\
 \varpi^{-i-j} (d_4y(d_1-1)+zb)  & \varpi^{-j}d_3 & d_4 & -\varpi^i b \\
 \varpi^{-2i-j}z(d_1d_4-1) & m&  & 1
 \end{bmatrix}$$ where $m=\varpi^{-i-j} (y(d_1-1) +zd_1b)$. Let $d_3\in \p^j$. If $y\neq 0$ we may assume $\val(y)<\val(z)-i$ and take $b\in \p^{-i}$, $d_4=d_1^{-1}\in 1+bz/y + \p^{i+j-\val(y)}$, so $R_g$ contains a conjugate of $U_K$. If $y=0$ and $z\neq 0$, we may assume $\val(z)<2i+j\leq i+j$ and take $b\in \p^{i+j-\val(z)}$, $d_3\in \p^j$, $d_4=1$ and $d_1\in 1+\p^{2i+j-\val(z)}$, so $R_g\supset U_S$.
\end{proof}

\begin{proposition}
Suppose $j\geq 0$ and $[t_{i,j}]\in \Supp(\pi)$. 

a) We have $i+j\leq n-1$. 

b) If $j=0$ then $2i\leq n-1$ and $R_{t_{i,0}}$ is the Klingen Levi subroup.

c) If $i\leq 0$ then $R_{t_{i,j}} = \begin{bmatrix}
                            * & * & 0 & 0 \\
                            0 & * & 0 & 0 \\
                            0 & * & *  & * \\
                            0 & 0 & 0 & *
                            \end{bmatrix}$. 
                            
d) If $i\geq 1$, $j\geq 1$, and $2i+j\leq n-1$ then $R_{t_{i,j}} = \begin{bmatrix}
                            * & 0  & 0 & 0 \\
                            0 & * & 0 & 0 \\
                            0 & * & *  & 0  \\
                            0 & 0 & 0 & *
                            \end{bmatrix}$
                            
e) If $i\geq 1$, $j\geq 1$, and $2i+j\geq n$ then $R_{t_{i,j}} = \begin{bmatrix}
                            * & 0 & 0 & 0 \\
                            0 & * & 0 & 0 \\
                            0 & * & *  &  0\\
                            * & 0 & 0 & *
                            \end{bmatrix}$

\end{proposition}

\begin{proof} Let $g=t_{i,j}$ and observe that $$g\Kl{n}g^{-1} = \begin{bmatrix}
 \OF & \p^i & \p^{i+j} & \p^{2i+j} \\
 \p^{n-i} & \OF & \p^j & \p^{i+j} \\
 \p^{n-i-j} & \p^{-j} & \OF& \p^i \\
 \p^{n-2i-j} & \p^{n-i-j} & \p^{n-i} & \OF 
\end{bmatrix}\cap G^0.$$ We have $2-n\leq i\leq n-1$ and $2i+j\geq 1$ by Lemma \ref{Main Support Restrictions}. In particular, $i+j\geq 1$ since $j\geq0$. 

a) Suppose $i+j\geq n$. Since $i\leq n-1$, we have $j\geq 1$. If $i\geq1$ then $2i+j\geq n$ and $R_g \supset U_S$. If $i\leq 0$ then $R_g \supset \begin{bsmallmatrix}
1 & *  & & \\
& 1 & & \\
* &* & 1 & * \\
& * & & 1 
\end{bsmallmatrix}$, a conjugate of $U_K$.

b) We have $i\geq 1$ by Lemma \ref{Main Support Restrictions}, so if $2i<n$ then $R_g$ is the Klingen Levi subgroup, while if $2i\geq n$ then $R_g \supset \begin{bsmallmatrix}
       1 & & \\
       & \SL_2(q) & \\
       * & & 1 
       \end{bsmallmatrix}$. Lemma \ref{Klingen Levi} then shows $\sigma^{R_g}=0$. 
       
c) We have $j\geq i+j\geq 2i+j\geq 1$ and $n-i\geq n-2i-j \geq n-i-j\geq 1$. So $R_g$ is as claimed. 
                            
d), e) These are clear. 
\end{proof}

We now consider double cosets with nondiagonal representative. By Proposition \ref{boundary cases imply diagonal} we may assume $i, j\geq 1$. Define the matrices \begin{eqnarray}
X_k &=& S(\varpi^k,0,0)\nonumber \\ 
Y_k &=& S(0, \varpi^k, 0)\nonumber \\ 
Z_k &=& S(0,0,\varpi^k)\nonumber .
\end{eqnarray} It is easy to see using integral diagonal matrices that $[t_{i,j} S(x,0,0)] =  [t_{i,j}X_{\val(x)}]$, $[t_{i,j} S(0,y,0)] =  [t_{i,j}Y_{\val(y)}]$, and $[t_{i,j} S(0,0,z)] =  [t_{i,j}Z_{\val(z)}]$. 

\begin{proposition} Let $1\leq j$, $1\leq i\leq n-1$, and $1\leq  k\leq n-1$. 

a) Suppose $k\leq i-1$. Then $[t_{i,j} X_k]\in \Supp(\pi)$ if and only if $i+j+k\leq n-1$, in which case $$R_{t_{i,j}X_k} = \left\{\begin{bmatrix}
a & & & \\
* & a & & \\
& & d & \\
* & & * & d
\end{bmatrix}: a, d\in \Fq^\times \right\}.$$

b) Suppose $ k\leq i+j-1$. Then $[t_{i,j} Y_k]\in \Supp(\pi)$ if and only if $2i+j\leq\min(n, 2k)-1$, in which case $$R_{t_{i,j}Y_k} = \left\{\begin{bmatrix}
a & & & \\
 & d & & \\
* & *  & a & \\
 & * &  & d
\end{bmatrix} : a, d\in \Fq^\times \right\}.$$

c)  Suppose $ k\leq 2i+j-1$. Then $[t_{i,j} Z_k]\in \Supp(\pi)$ if and only if $i+j+1\leq k$, in which case $$R_{t_{i,j}Z_k} = \left\{\begin{bmatrix}
a & & & \\
& ad & & \\
& * & a/d & \\
* & & & a
\end{bmatrix}: a, d\in \Fq^\times \right\}.$$
\end{proposition}

\begin{proof}
a) Let $g = t_{i,j}X_k$ and compute $$g \begin{bmatrix}
 a & & & \\
& d_1 & & \\
 & d_3  & a d_4 & \\ 
\varpi^n c  &  & & d_1 d_4
\end{bmatrix} g^{-1} = \begin{bmatrix}
a & & & \\
\varpi^{-i+k}(a-d_1) & d_1 & & \\
-\varpi^{-i-j+k}d_3 & \varpi^{-j}d_3 & ad_4 & \\
\varpi^{-2i-j}(\varpi^nc + \varpi^{2k}d_3) & -\varpi^{-i-j+k}d_3 & m & d_1d_4
\end{bmatrix}
$$ where $m=\varpi^{-i+k}d_4(a-d_1)$. Suppose $i+j+k\geq n$, so $2i+j>i+j+k\geq n$. Taking $d_1=d_4=1$, $a\in 1+\p^{i-k}$, $d_3\in \p^{i+j-k}\subset  \p^{j+1}$, and $c \in -\varpi^{2k-n}d_3 + \p^{2i+j-n}\subset \p^{i+j+k-n}$ shows that $R_g\supset U_K$. Conversely, suppose $1\leq k\leq n-i-j-1$. Taking $c=0$, $d_3\in \p^{2i+j-2k} \subset \p^{i+j-k+1}\subset \p^{j+1}$, and $a\in d_1+\p^{i-k}$ shows that $R_g$ contains the claimed subgroup. To see that this inclusion is an equality, observe first that $2k \leq i-1 + n-i-j-1=n-j-2$. If $$h = \begin{bmatrix}
a_1 & b_1 & b_2 & b_3 \\
\varpi^n c_1 & d_1 & d_2 & b_4\\
\varpi^n c_2 & d_3 & d_4 & b_5 \\
\varpi^n c_3 & \varpi^n c_4 & \varpi^n c_5 & a_2
\end{bmatrix}\in \Kl{n}$$ then $ghg^{-1}\in K$ if and only if the following conditions hold: 
\begin{eqnarray}
d_3&\in& \p^{i+j-k}\nonumber \\
 a_1-d_1-\varpi^{k}b_1&\in& \p^{i-k} \nonumber \\
a_2-d_4-\varpi^{k}b_5&\in&  \p^{i-k} \nonumber \\
 \varpi^{n-2k}(c_3-\varpi^kc_2-\varpi^kc_4)+d_3&\in& \p^{2i+j-2k} \nonumber. 
\end{eqnarray}
Since $i+j-k < \min(n-2k, 2i+j-2k)$, the last condition will fail if $\val(d_3)\leq i+j-k$. Thus we must have $d_3\in \p^{i+j-k+1}\subset \p^{j+1}$ and the conclusion follows from the form of $ghg^{-1}$. 

b) Let $g = t_{i,j}Y_k$ and compute $$g\begin{bmatrix}
a & & & \\
& d_1 &  d_2 & \\
& d_3 &  d_4 & \\
\varpi^n c & & & \Delta/a
\end{bmatrix}g^{-1} = \begin{bmatrix}
a & & & \\
-\varpi^{-i+k}d_2 & d_1 & \varpi^j d_2 & \\
\varpi^{-i-j+k}(a-d_4) & \varpi^{-j }d_3 & d_4 & \\
\varpi^{-2i-j}( \varpi^n c - \varpi^{2k}d_2 ) & \varpi^{-i-j+k}(d_1-\Delta/a) & \varpi^{-i+k}d_2 & \Delta/a
\end{bmatrix}$$ where $\Delta = d_1d_4 -d_2d_3$. Let $d_3\in \p^j$, $a\in d_4+\p^{i+j-k}$. If $2i+j\geq n$ we take $d_1=1$, $d_2=0$ and $c\in \p^{2i+j-n}$ so $R_g\supset U_S$. If $2i+j\geq 2k$ we take $d_1=1$, $c=0$ and $d_2\in \p^{2i+j-2k}\subset \p^{i-k+1}$ so $R_g\supset U_S$. Conversely, suppose $2i+j\leq \min(n, 2k)-1$. Taking $c=d_2=0$, 
shows that $R_g$ contains the claimed subgroup. To see that equality holds, let $h\in \Kl{n}$ as in the previous part. Then $ghg^{-1}\in K$ if and only if the following conditions hold: \begin{eqnarray}
 a_1-d_4&\in& \p^{i+j-k} \nonumber \\
d_1-a_2&\in&  \p^{i+j-k} \nonumber \\
d_3-\varpi^k(b_1-b_5)&\in&  \p^{j}. \nonumber 
\end{eqnarray} The conclusion follows from the form of $ghg^{-1}$. 

c) Let $g = t_{i,j}Z_k$ and compute $$g \begin{bmatrix}
a & -ab & & \\
& d_1 & & \\
& d_3 & ad_4 & ad_4 b\\
& & & d_1d_4
\end{bmatrix} g^{-1} = \begin{bmatrix}
a & -\varpi^iab & & \\
& d_1 & & \\
-\varpi^{-i-j+k}ad_4b & \varpi^{-j}d_3 & ad_4 & \varpi^i ad_4 b \\
\varpi^{-2i-j+k}(a-d_1d_4) & -\varpi^{-i-j+k}ab & & d_1d_4
\end{bmatrix}.$$ Suppose that $k\leq i+j$. Taking $b\in \p^{i+j-k}$, $d_3\in \p^j$, $d_1=d_4=1$, and $a\in 1+\p^{2i+j-k}$ shows $R_g\supset U_S$. Conversely, suppose $i+j+1\leq k$. Taking $b=0$, $d_3\in \p^j$ and $a\in d_1d_4+\p^{2i+j-k}$ shows that $R_g$ contains the claimed subgroup. To see that equality holds, let $h\in \Kl{n}$ as in the previous part. Then $ghg^{-1}\in K$ if and only if the following conditions hold: \begin{eqnarray}
 d_3 &\in& \p^{j} \nonumber \\
\varpi^{n-k}c_3 + a_1-a_2 -\varpi^kb_3 &\in&  \p^{2i+j-k}. \nonumber 
\end{eqnarray} The form of $R_g$ now follows from the form of $ghg^{-1}$.  
\end{proof}

Now we will rule out some remaining double cosets from lying in $\Supp(\pi)$. For the inequalities appearing here, see Lemma \ref{double coset constraints}. 

\begin{lemma} Suppose $1\leq j$,  $1\leq \val(x)\leq i-1$, and $1\leq \val(y)\leq i+j-1$. 
 
a) If $1\leq \val(y/x)\leq j-1$ then $[t_{i,j} S(x,y,0)] \not\in \Supp(\pi)$.  

b) If $1\leq \val(z/x)\leq i+j-1$ then $[t_{i,j} S(x,0,z)] \not\in \Supp(\pi)$.  

c) If $1\leq \val(z)\leq 2i+j-1$ then $[t_{i,j} S(0,y,z)] \not\in \Supp(\pi)$.  
\end{lemma}

\begin{proof}
a) Let $g=t_{i,j}S(x,y,0)$, $d_3\in \p^j$, $m_2\in \p^{i+j}$, and $m_3\in \p^{2i+j}$. Set $d_1 = 1+\frac{m_2}{y}+\frac{d_3 x}{y}\in 1+\p$, $d_4 = 1+\frac{m_2}{y}+ \frac{m_3}{xy}\in 1+\p$ and compute $$g \begin{bmatrix}
d_1d_4 & & & \\
& d_1 & & \\
& d_3 & d_4 & \\
& & & 1
\end{bmatrix} g^{-1} = \begin{bmatrix}
d_1d_4 & & & \\
d_1m & d_1 & & \\
\varpi^{-i-j}(d_1m_2+(d_1-1)\frac{m_3}{x}) & \varpi^{-j}d_3 & d_4 & \\
\varpi^{-2i-j}m_3 & \varpi^{-i-j}m_2 & -m & 1
\end{bmatrix}$$ where $m=\frac{m_3}{\varpi^i y}+\frac{m_2 x}{\varpi^i y}\in \p$. Thus $R_g\supset U_S$.


b) Let $g=t_{i,j}S(x,0,z)$ and compute $$g \begin{bmatrix}
d_1d_4 & & & \\
& d_1 & & \\
& d_3 & d_4 & \\
& & & 1
\end{bmatrix} g^{-1} = \begin{bmatrix}
d_1d_4 & & & \\
\varpi^{-i}xd_1(d_4-1) & d_1 & & \\
-\varpi^{-i-j}xd_3 & \varpi^{-j}d_3 & d_4 & \\
\varpi^{-2i-j}( z(d_1d_4-1) +x^2d_3 ) & -\varpi^{-i-j}  xd_3 & \varpi^{-i}x(1-d_4) & 1
\end{bmatrix}.$$ Taking $d_3\in \p^{i+j-\val(x)}\subset \p^{j+1}$, $d_4\in 1+\p^{i-\val(x)}$, and $d_1\in d_4^{-1} \left(1 -x^2d_3/z \right) + \p^{2i+j-\val(z)}$ shows $R_g\supset U_K$. 

c)  Let $g=t_{i,j}S(0,y,z)$ and compute $$g \begin{bmatrix}
d_1d_4 & & & \\
& d_1 & & \\
& d_3 & d_4 & \\
& & & 1
\end{bmatrix} g^{-1} = \begin{bmatrix}
d_1d_4 & & & \\
 & d_1 & & \\
\varpi^{-i-j}yd_4(d_1-1) & \varpi^{-j}d_3& d_4 & \\
\varpi^{-2i-j} z(d_1d_4-1)  & \varpi^{-i-j}y(d_1-1) &  & 1
\end{bmatrix}.$$ Taking $d_3\in \p^{j}$, $d_1\in 1+\p^{i+j-\val(y)}$, and $d_4\in d_1^{-1} + \p^{2i+j-\val(z)}$ shows $R_g\supset U_S$. 
\end{proof}

There is one remaining family of double cosets to consider. 

\begin{lemma} \label{last family of double cosets}
Suppose $1\leq \val(x)<i$, $1\leq \val(y/x)< j$, and $1\leq \val(z/y)$. Let $g = t_{i,j}S(x,y,z)$. 

a) We have $[g]\in \Supp(\pi)$ if and only if the following conditions hold: 

i) $i+j< i + \val(y+z/x) =j+\val(xz/y) <n$. 




ii) $j<2\val(y/x)$. 

b) If conditions i) and ii) hold then $R_g$ is conjugate (by a diagonal matrix) to $$\left\{ \begin{bmatrix}
a & & & \\
v & a & & \\
b & & a & \\
 c & b & -v & a
\end{bmatrix}\begin{bmatrix}
1 & & & \\
 & 1 & & \\
& v & 1 & \\
& &  & 1
\end{bmatrix} : a\in \Fq^\times, b, c, v\in \Fq \right\}.$$
\end{lemma}

\begin{proof} Suppose first that $[g]\in \Supp(\pi)$.

Proof of i): Let $m_1\in \p^i$, $m_2\in \p^{i+j}$, $d_3\in \p^{j}$, $d_1= 1-\frac{m_1}{x}$, $d_4=1-\frac{m_2}{y} - \frac{d_3 x}{y}$, and compute 
$$g\begin{bmatrix}
1 & & & \\
& d_1 &  & \\
& d_3 & d_4 & \\
 & & & d_1d_4
\end{bmatrix}g^{-1}  = \begin{bmatrix}
1 & & & \\
\varpi^{-i}m_1 & d_1 &  & \\
\varpi^{-i-j}m_2 & \varpi^{-j}d_3 & d_4 & \\
\varpi^{-2i-j}m_3 & \varpi^{-i-j}m_4 & -\varpi^{-i}d_4m_1 & d_1 d_4
\end{bmatrix}$$ where \begin{eqnarray}
m_3 &=& m_2\left(\frac{z}{y}-x\right) - \frac{m_1m_2z}{xy}  + \frac{d_1d_3xz}{y} +  m_1\left(y+ \frac{z}{x}\right) \nonumber\\
m_4 &=& m_2 - \frac{m_1(m_2+xd_3)}{x}\nonumber.
\end{eqnarray}

 If $\val(z)>\val(xy)$ and $i < j + \val(zx/y^2) $, then taking  $$m_1 \in \left(1+\frac{z}{xy}(1 - \frac{d_3 x}{y} - m_2) \right)^{-1}  \left( \frac{m_2 x}{y}(1-\frac{z}{xy}) - \frac{d_3 xz}{y^2} \right) + \p^{2i+j-\val(y)}$$ yields $m_1\in \p^{i+1}$ and so $R_g \supset U_S$. Note that if $\val(z/y)\geq i$ then $\val(z)>\val(xy)$ and $i<j+ \val(zx/y^2)$ since $\val(x)<i$ and $j + \val(zx/y^2) = j - \val(y/x) + \val(z/y) >i$. So we must have $\val(z/y)<i$. 

On the other hand, if $\val(z)>\val(xy)$ and $j < i + \val(y^2/xz)$ then taking $$d_3 \in  \left(1-\frac{m_1}{x}\right)^{-1} \left( \frac{m_1m_2}{x^2}-\frac{m_2 x}{y} \frac{y^2}{xz}(1-\frac{z}{xy})  - \frac{y^2 m_1}{xz}(1+\frac{z}{xy})  \right) + \p^{2i+j +\val(y/xz)}$$ yields $d_3\in \p^{j+1}$ so $R_g \supset U_K$. Therefore if $\val(z)> \val(xy)$ we must have $i +\val(y)= j + \val(xz/y)$. 

If $\val(z)<\val(xy)$ and $i<j+\val(x^2/y)$ then taking $$m_1 \in \left( 1+\frac{yx}{z}-\frac{d_3 x+m_2}{y}  \right)^{-1}\left( - \frac{m_2 x}{y}(1-\frac{xy}{z}) -\frac{d_3 x^2}{y}   \right) + \p^{2i+j+\val(x/z)}$$ yields $m_1\in \p^{i+1}$ and so $R_g \supset U_S$.

On the other hand, if $\val(z)<\val(xy)$ and $j<i+\val(y/x^2)$ then taking $$d_3 \in  \left(1-\frac{m_1}{x}\right)^{-1} \left( \frac{m_1m_2}{x^2}-\frac{m_2}{x} (1-\frac{xy}{z})  - \frac{m_1y}{x^2}(1+\frac{xy}{z})  \right)  +\p^{2i+j +\val(y/xz)}$$ yields $d_3\in \p^{j+1}$ so $R_g \supset U_K$. Therefore if $\val(z) <\val(xy)$ we must have $i+\val(y) = j+\val(x^2)$. 

If $\val(z)=\val(xy)$ and $i + \val(1+\frac{z}{xy}) < j + \val(x^2/y)  = j+\val(xz/y^2)$ then taking $$m_1 \in \left( 1- (1+\frac{z}{xy})^{-1}( \frac{d_3z}{y^2}+\frac{m_2 z}{xy^2} )  \right)^{-1}  (1+\frac{z}{xy})^{-1}  \left(    \frac{m_2x }{y}(1-\frac{z}{xy}) -\frac{d_3 xz}{y^2}   \right) + \p^{2i+j-\val(y+\frac{z}{x}) }$$ yields $m_1\in \p^{i+1}$ and so $R_g \supset U_S$.

On the other hand, if $\val(z)=\val(xy)$ and $j<i+\val(y/x^2) + \val(1+\frac{z}{xy})$  then taking $$d_3 \in  \left(1-\frac{m_1}{x}\right)^{-1} \left( \frac{m_1m_2}{x^2}-\frac{m_2}{x} (1-\frac{xy}{z})  - \frac{m_1y^2}{xz}(1+\frac{z}{xy})  \right)  +\p^{2i+j +\val(y/xz)}$$ yields $d_3\in \p^{j+1}$ so $R_g \supset U_K$. Therefore if $\val(z) = \val(xy)$ we must have $i+\val(y) + \val(1 + \frac{z}{xy}) = j+\val(x^2)$. 

We have shown that $i + \val(y+z/x) =j+\val(xz/y)$ must hold in all cases. Suppose next that $i\geq \val(xz/y)$. Then $j\geq \val(y+z/x) $. The proof so far makes clear that $$R_g \supset  \left\{ \begin{bmatrix}
a & & & \\
& a & & \\
b & & a & \\
 c & b & & a
\end{bmatrix} : a\in \Fq^\times, b, c\in \Fq \right\}.$$ Let $m\in \p^j$, $b = \frac{xm} {xy+z}$, $d_1= 1-bx$ and $d_3 = d_1^{-1}(m \left(1-\frac{bz}{y}\right) + b^2 z + b y(b x-2))$. We compute $$g\begin{bmatrix}
1 & b & & \\
& d_1 & & \\
& d_3 & 1 & -b \\
& & & d_1 
\end{bmatrix} g^{-1} = \begin{bmatrix}
d_1 & \varpi^i b & & \\
& 1 & &  \\
& \varpi^{-j}m & d_1 & -\varpi^i b\\
& & & 1
\end{bmatrix}$$ and conclude that $R_g\supset U_S$.

Finally, suppose $n\leq j+\val(xz/y) $. Then taking $d_3\in \p^j$, $d_4=1 - \frac{d_3 x}{y}$ as above,  we compute $$g\begin{bmatrix}
1 & & & \\
& 1 & & \\
& d_3 & d_4 & \\
 \frac{-d_3xz}{y}   & & & d_4
\end{bmatrix}g^{-1}= \begin{bmatrix}
1 & & & \\
& 1 & & \\
 & \varpi^{-j} d_3 & d_4 & \\
 &  &  & d_4
 \end{bmatrix}$$ and conclude that $R_g\supset U_S$.


Proof of ii): Suppose $\val(y^2/x^2)\leq  j$. Take $d\in \p^j$ and compute $$g\begin{bmatrix}
1 & & & \\
& 1+ dx/y & -d x^2/y^2 & \\
& d & 1- d x/y & \\
 & & & 1
\end{bmatrix}g^{-1} = \begin{bmatrix}
1 & & & \\
& 1+ dx/y & -\varpi^{j} d x^2/y^2  & \\
 & \varpi^{-j}d & 1- d x/y & \\
 &  &   & 1
\end{bmatrix}.$$ Together with the above, this yields $R_g\supset U_S$. 

Conversely, suppose that i) and ii) hold. Take $m\in \p^i$, $d_1 = 1-\frac{m}{x}\in 1+\p$, $d_3 = -m(y+\frac{z}{x})\frac{y}{d_1xz}  \in \p^j$, and $d_4 = 1-\frac{d_3x}{y}\in 1+\p$. Then $$g\begin{bmatrix} 1 & & & \\
& d_1 & & \\
& d_3 & d_4 & \\
& & & d_1d_4
\end{bmatrix}g^{-1} = \begin{bmatrix}
1 & & & \\
\varpi^{-i}m & d_1 & & \\
0 & \varpi^{-j} d_3  & d_4 & \\
 0 & -\varpi^{-i-j}md_3 & -\varpi^{-i}d_4m  & d_1 d_4
\end{bmatrix}$$ which shows that $R_g$ contains a conjugate of the claimed subgroup. We will show that equality holds. Then Lemma \ref{Last character computation} below shows that $\dim \sigma^{R_g}>0$. 

Observe that $\val(z)\geq i+1 + \val(y/x) \geq i+1 + \frac{j+1}{2}$ and $\val(y) \geq \frac{j+1}{2}+\val(x)$ so $\val(yz)>i+j$, $\val(y^2)>j$, $\val(z^2)>2i+j$ and $\val(xz)>i$. It follows that 
$$g\begin{bmatrix}
1 & & & b \\
& 1 &  &\\
& & 1 & \\
& & & 1
\end{bmatrix}g^{-1} \in K^+$$ for all $b\in \OF$. Since $i+j<n$, we also have $$g\begin{bmatrix}
1 & & & \\
\varpi^n c_1 & 1 & & \\
\varpi^n c_2 & & 1 & \\
& \varpi^n c_2 & -\varpi^n c_1 & 1
\end{bmatrix}g^{-1}\in \begin{bmatrix}
1 & & & \\
& 1 & & \\
& & 1 & \\
\p^{n-2i-j}  & & & 1
\end{bmatrix}K^+$$ for all $c_1, c_2\in \OF$. So to determine $R_g$ we need to consider $$h:= g \begin{bmatrix}
1 & & & \\
& d_1 & d_2 & \\
& d_3 & d_4 & \\
\varpi^n c & & & d_1 d_4 - d_2 d_3 
\end{bmatrix}\begin{bmatrix}
1 & b_1 & b_2 & \\
& 1 & & b_2 \\
& & 1 & -b_1 \\
& & & 1
\end{bmatrix} g^{-1}.$$

Suppose first that $\val(y)> j$. The condition $h
\in K$ forces $d_3\in \p^j$, as well as the existence of $m_1\in \p^i$ and $m_2\in \p^{i+j}$ such that \begin{eqnarray}
d_1 &=& 1 -b_1 x-b_2 y -\frac{d_2y}{x} - \frac{m_1}{x} \nonumber\\
d_4 &=& \left(1-\frac{b_1 z}{y}\right)^{-1}(1-b_1 x - b_2 y - \frac{d_3 x}{y} -\frac{m_2}{y}  ) .\nonumber
\end{eqnarray} One then only needs the $(4,1)$ entry in $h$ to lie in $\OF$. It can be checked that this forces $\val(\varpi^{-i}m_1)=\val(\varpi^{-j}d_3)$, and the form of $R_g$ follows. 

Now suppose that $\val(y) \leq j$. This forces $\val(z)=\val(xy)$ and $\val(1+\frac{z}{xy})>0$, which also implies $\val(x^2)>i$ and $\val(y^2)>i+j$.  The condition $h\in K$ then forces the existence of $m_6\in \p^j$ and $m_1, m_2$ as above with 
\begin{eqnarray}
d_1 &=& 1  -\frac{d_2y}{x} - \frac{m_1}{x} \nonumber\\
d_4 &=& \left(1-\frac{b_1 z}{y}\right)^{-1}(1-b_1 x  - \frac{d_3 x}{y} -\frac{m_2}{y}  ) \nonumber \\
d_3 &=& \left(1-b_2 y\right)^{-1}(m_6 - b_1y (1+d_4)). \nonumber 
\end{eqnarray} One then only needs the $(4,1)$ entry in $h$ to lie in $\OF$. It can be checked that this forces $\val(\varpi^{-i}m_1)=\val(\varpi^{-j}d_3)$, and the form of $R_g$ follows. 
\end{proof}

We now make the following observation.

\begin{theorem}\label{Nongenerics from K}
Suppose $\pi = \cInd_{ZK}^G(\sigma)$ is a nongeneric depth zero supercuspidal irreducible representation of $G$. Then $\pi^{\Kl{n}}=0$ for all $n\geq 0$.
\end{theorem}

\begin{proof}
The condition on $\pi$ is equivalent to $\sigma$ being a nongeneric cuspidal irreducible representation of $\GSp(4, \Fq)$. For all the subgroups $R_g$ associated to potential $g\in \Supp(\pi)$, we have seen that $R_g$ contains a conjugate of the subgroup $\left\{ \begin{bsmallmatrix}
1 & & & x \\
& 1 & & \\
& & 1 & \\
& & & 1
\end{bsmallmatrix}: x\in \Fq \right\}$. So $\sigma^{R_g}=0$ by Lemma \ref{nongeneric character lemma}.
\end{proof}

\begin{corollary}
Suppose $\pi$ is a nongeneric depth zero supercuspidal irreducible representation of $G$. Then $\pi^{\Kl{n}}=0$ for all $n\geq 0$.
\end{corollary}

\begin{proof}
This follows from Theorems \ref{Nongeneric Main Theorem} and \ref{Nongenerics from K}. 
\end{proof}

\subsection{Counting double cosets}

We have determined the necessary and sufficient conditions on $[t_{i,j}S(x,y,z)]$ to lie in $\Supp(\pi)$. It remains to consider when two such double cosets are equal, and then to enumerate $\Supp(\pi)$. 

The first column of Table 1 below lists the families of double coset representatives. The conjugacy class of $R_g$ in $\GSp(4, \Fq)$ is an invariant of $[g]$. From the $R_g$ column in Table 1, one sees that the only potential equalities of double cosets occur within a row or else between those of the form $[t_{i,j}X_k]$ and $[t_{i',j'}Y_{k'}]$. Suppose now that $j\geq 0$. If $i\leq 0$ then it is not hard to verify that $[t_{i,j}]$ determines $i$ and $j$ directly; see also Remark \ref{paramodular remark}. So assume $i\geq 1$. By the Cartan decomposition and the fact that $\Kl{n}\subset K$, the double coset $[t_{i,j}S(x, y, z)]$ determines $i$ and $j$. This takes care of rows 2, 3 and 4 of Table 1. It is easy to check by a direct matrix computation that, for the parameters appearing in rows 5, 6, and 7, $[t_{i,j}X_k]$, $[t_{i,j}Y_k]$, and $[t_{i,j}Z_k]$ each determine $k$ uniquely as well. It is also not hard to show directly that $[t_{i,j}X_k]\neq [t_{i,j}Y_{k'}]$, so there is no overlap between rows 6 and 7. 

\begin{remark} \label{paramodular remark}
The functions supported on the double cosets $ZKt_{i,j}\K{n}$, where $\K{n}$ is the paramodular group of level $\p^n$, and $t_{i,j}$ comes from the first row of Table 1, provide a full set of basis vectors for the space of $\K{n}$-fixed vectors in $\pi$. Since $\K{n}\supset \Kl{n}$, this provides a different way to show that $[t_{i,j} ]$ determines these $i$ and $j$. 
\end{remark}

This just leaves the last family of double cosets $[g]=[t_{i,j}S(x,y,z)]$ considered in Lemma \ref{last family of double cosets}. It is easy to show by direct computation that $[g]$ determines $\val(x)$, $\val(y)$, and $\val(z)$. Using integral diagonal matrices, we may assume that $x$ and $y$ are powers of $\varpi$. So it remains only to consider what conditions on a unit $u\in \OF^\times$ are necessary and sufficient to allow $[t_{i,j}S(x, y, uz )] = [t_{i,j}S(x, y, z )].$

\begin{lemma} \label{distinction for S(k)}
Let $[g]:=[t_{i,j} S(x, y, uz)]$ and $[g_1]:=[t_{i,j} S(x, y, z)]$ lie in $\Supp(\pi)$, where $[g_1]$ satisfies the conditions of Lemma \ref{last family of double cosets} and $u\in \OF^\times$. Then $[g]=[g_1]$ if and only if $u\in 1+\p^{j-\val(y/x)}$. 
\end{lemma}

\begin{proof}
If $\val(u-1) \geq j-\val(y/x)$ then $[g_1]=[g]$ 
since $$g_1 \begin{bmatrix}
1 & & & \\
& 1 & & \\
& \frac{y}{x}(1-u) & u & \\
 & & & u
\end{bmatrix} g^{-1} = \begin{bmatrix}
1 & & & \\
& 1 & & \\
& \varpi^{-j}\frac{y}{x}(1-u) & u & \\
 & & & u
\end{bmatrix} \in K.$$   

If $\val(u-1)<j-\val(y/x)$ then one can show by a tedious matrix computation that $[g_1]\neq[g]$; we omit the details.\end{proof}

For $\vec{k}=(i, k_x, k_y, k_z, u)\in \Z_+^4\times \OF^\times$ let $j=i+\val(\varpi^{k_y}+u\varpi^{k_z-k_x})-(k_x+k_z-k_y)$; we assume that $\varpi^{k_y}+u\varpi^{k_z-k_x}\neq 0$. Then define  $$S({\vec{k}}) = t_{i,j}S(\varpi^{k_x}, \varpi^{k_y}, u\varpi^{k_z})\in G.$$ We have determined the conditions on $\vec{k}$ so that $[S(\vec{k})]\in \Supp(\pi)$, as well as necessary and sufficient conditions for two such double cosets to be distinct. It remains to compute the number of such double cosets. First, we do so for the others.

\begin{proposition}
The numbers of distinct double cosets of each type, except those of the form $[S(\vec{k})]$, are given in the last column of Table 1. 
\end{proposition}
\begin{proof}
The numbers of double cosets $[t_{i,j}]$ for the first four rows of Table 1 are easy to compute.


Consider the double cosets $[t_{i,j}X_k]$. The conditions on $(i,j,k)$ can be written as $k+1\leq i\leq n-1-j-k$, $1\leq j\leq n-2-2k$, $1\leq k \leq \left\lfloor \frac{n-3}{2} \right\rfloor$. The reader can check that $$\sum\limits_{k=1}^{\left\lfloor \frac{n-3}{2} \right\rfloor} \sum\limits_{j=1}^{n-2-2k}  (n-1-j-2k)  = \left\lfloor   \frac{(n-1)(n-3)(2n-7)}{24} \right\rfloor.$$ 

It is not hard to check that the map $(i,j,k)\mapsto (i,j,i+j+k)$ induces a bijection between the sets of $[t_{i,j}X_k]$ and $[t_{i,j}Z_k]$ lying in $\Supp(\pi)$. Equality with the number of $[t_{i,j}X_k]$ follows. 

For $[t_{i,j}Y_k]$, the conditions $1\leq j$, $1\leq i$, $k\leq i+j-1$, and $2i+j\leq \min(n, 2k)-1$ force $i+\left\lceil \frac{j+1}{2} \right\rceil \leq k$,  $1\leq i\leq \left\lfloor \frac{n-j-1}{2} \right\rfloor$, and $3\leq j \leq n-3$. The reader can check that   $$\sum\limits_{j=3}^{n-3}\sum\limits_{i=1}^{\left\lfloor \frac{n-j-1}{2} \right\rfloor} \sum\limits_{k=i+\left\lceil \frac{j+1}{2} \right\rceil}^{i+j-1} (1)=\frac{n-3}{6}\left\lfloor\frac{(n-2)(n-4)}{4} \right\rfloor.$$

This completes the proof. 
\end{proof}

We finally count the last family of double cosets, those of the form $[g]=[S(\vec{k})]$. For ease of notation, let $x=\varpi^{k_x}$, $y=\varpi^{k_y}$ and $z=u\varpi^{k_z}$. For a fixed $(i, k_x,k_y,k_z)$, the distinct $[S(\vec{k})]$ correspond to orbits of $1+\p^{j-\val(y/x)}$ under multiplication on the sets $$\begin{cases}
\OF^\times  & \text{ if } \val(z)\neq\val(xy)\\
\OF^\times\setminus (-1+\p)& \text{ if } \val(z)=\val(xy)  \text{ and } \val(1+u)=0 \\
(-1+\p^{\val(1+u)})\setminus (-1+\p^{\val(1+u)+1})  & \text{ if } \val(z)=\val(xy)  \text{ and } \val(1+u)>0 
\end{cases}.$$ This follows from Lemma \ref{distinction for S(k)}. So the number of choices of $u$ yielding distinct $[S(\vec{k})]$ are $$\begin{cases}
(q-1)q^{j-\val(y/x)-1}  & \text{ if } \val(z)\neq\val(xy)\\
(q-2)q^{j-\val(y/x)-1} & \text{ if } \val(z)=\val(xy)  \text{ and } \val(1+u)=0 \\
(q-1)q^{j-\val(y/x)-\val(1+u)-1} & \text{ if } \val(z)=\val(xy)  \text{ and } \val(1+u)>0 
\end{cases}.$$

Note $j-\val(y/x)-\val(1+u)=i-\val(x)>0$ when $\val(z)=\val(xy)$. Let $k=j-\val(y/x)+1$.

\begin{lemma} \label{number with z<xy}
a) The number of double cosets $[S(\vec{k})]\in \Supp(\pi)$ with $\val(z)<\val(xy)$ is \begin{equation}
q^{\left\lfloor \frac{n-5}{4} \right\rfloor}a_q\left(n-4\left\lfloor \frac{n-5}{4} \right\rfloor\right) - a_q(n)\nonumber 
\end{equation}
where $$a_q(n)= \left\lfloor \frac{(n-1)(n-3)(2n-7)}{24} \right\rfloor+\frac{n^2-n+1}{q-1}+\frac{8n+12}{(q-1)^2}+\frac{32}{(q-1)^3}.$$ 

b) The number of double cosets $[S(\vec{k})]\in \Supp(\pi)$ with $\val(z)>\val(xy)$ is the same as the number with $\val(z)<\val(xy)$.  
\end{lemma}
\begin{proof} a) Let $a=\val(x)-k$, $b=\val(y)-2k$ and $c=\val(z)-3k$. The system of inequalities becomes 
$c-b+1\leq  a\leq b$, $\left\lfloor \frac{c+2}{2}\right\rfloor \leq  b\leq c$, 
$1\leq  c\leq n-4k$, 
$\leq  k\leq \left\lfloor \frac{n-1}{4}\right\rfloor=:M$. Let $A=\left\lfloor \frac{(n-5)(n-7)(2n-15)}{24} \right\rfloor$. The number of double cosets is $q-1$ times
\begin{eqnarray}
& &\sum\limits_{k=2}^{M} q^{k-2} \sum\limits_{c=1}^{n-4k}\sum\limits_{b=\left\lfloor \frac{c+2}{2}\right\rfloor }^{c}(2b-c) \nonumber \\
&=& \sum\limits_{k=2}^{M} q^{k-2} \sum\limits_{c=1}^{n-4k}\left\lfloor \frac{c+1}{2} \right\rfloor \left\lfloor \frac{c+2}{2} \right\rfloor\nonumber \\
&= &  \sum\limits_{k=2}^M q^{k-2}  \left\lfloor \frac{(n-4k+1)(n-4k+3)(2n-8k+1)}{24}\right\rfloor    \nonumber\\
&= &  \sum\limits_{m=0}^{M-2} q^{m}  \left\lfloor \frac{(n-7-4m)(n-5-4m)(2n-15-8m)}{24}\right\rfloor    \nonumber \\
&=&\sum\limits_{m=0}^{M-2} q^m\left(A - (n^2-17n+73)m+ (4n-42)m(m-1)  -\frac{16}{3}m(m-1)(m-2) \right)\nonumber\\
&=&\left(A - (n^2-17n+73)q\frac{d}{dq}+ (4n-42)q^2\frac{d^2}{dq^2} -\frac{16}{3}q^3\frac{d^3}{dq^3} \right)\left( \frac{q^{\left\lfloor \frac{n-5}{4}\right\rfloor}-1}{q-1}  \right)\nonumber.
\end{eqnarray} The remainder of the computation is tedious but elementary. 

b)  Let $a=\val(x)+k$, $b=\val(y)$ and $c=\val(z)+k$. 
 The system of inequalities becomes $a+b+1\leq c\leq n $, 
$a\leq b\leq n-1-a$, $2k\leq a\leq \left\lfloor \frac{n-1}{2}\right\rfloor$, $2\leq k\leq \left\lfloor \frac{n-1}{4}\right\rfloor=:M$. The number of double cosets is $q-1$ times \begin{eqnarray}
 \sum\limits_{k=2}^M q^{k-2} \sum\limits_{a=2k}^{\lfloor \frac{n-1}{2}\rfloor} \sum\limits_{b=a}^{n-1-a}(n-a-b) 
 & = &  \sum\limits_{k=2}^M q^{k-2} \sum\limits_{a=2k}^{\lfloor \frac{n-1}{2}\rfloor} \sum\limits_{b=1}^{n-2a} b \nonumber \\
 &=& \sum\limits_{k=2}^M q^{k-2} \sum\limits_{a=2k}^{\lfloor \frac{n-1}{2}\rfloor} \frac{(n-2a)(n-2a+1)}{2}. \nonumber 
 \end{eqnarray} The internal sum evaluates to the same quantity as appeared in the case $\val(z)<\val(xy)$.
\end{proof}

\begin{lemma}\label{number with z=xy and 1+z/xy a unit}
The number of double cosets $[S(\vec{k})]\in \Supp(\pi)$ with $\val(z)=\val(xy)$ and $\val(1+u)=0$ is \begin{equation}
q^{\left\lfloor \frac{n-4}{4} \right\rfloor}b_q\left(n-4\left\lfloor \frac{n-4}{4} \right\rfloor \right) - b_q(n)\nonumber 
\end{equation}
where $$b_q(n)= \left\lfloor \frac{(n-2)^2}{4} \right\rfloor-\frac{1}{q-1}\left\lfloor \frac{n^2-12n+4}{4}\right\rfloor+\frac{8-2n}{(q-1)^2}-\frac{8}{(q-1)^3}.$$ 
\end{lemma}
\begin{proof} Let $a=\val(x)+k$ and $b=\val(y)$. Then the inequalities become $a\leq b \leq n-a$, $2k\leq a \leq \left\lfloor \frac{n}{2}\right\rfloor$, $2\leq k\leq \left\lfloor \frac{n}{4}\right\rfloor=:M$, and the number of double cosets is $q-2$ times 
\begin{eqnarray}
\sum\limits_{k=2}^M q^{k-2} \sum\limits_{a=2k}^{\left\lfloor \frac{n}{2} \right\rfloor} (n-2a+1) 
&=& \sum\limits_{k=2}^M q^{k-2} \left( \left\lfloor \frac{n^2+4n+4}{4} \right\rfloor -2(n+2)k +4k^2\right) \nonumber \\
&=&\sum\limits_{m=0}^{\lfloor \frac{n-8}{4}\rfloor} q^m\left(   \left\lfloor \frac{n^2-12n+36}{4} \right\rfloor +(16-2n)m+  4m(m-1) \right)\nonumber\\
&=& \left(\left\lfloor \frac{n^2-12n+36}{4} \right\rfloor+ (16-2n)q\frac{d}{dq}+4q^2\frac{d^2}{dq^2}   \right)\left( \frac{q^{\lfloor \frac{n-4}{4}\rfloor} -1 }{q-1}  \right) \nonumber.
\end{eqnarray} The remainder of the computation is tedious but elementary. \end{proof}

\begin{lemma}\label{number with z=xy and 1+z/xy not a unit}
The number of double cosets $[S(\vec{k})]\in \Supp(\pi)$ with $\val(1+u)>0$ is \begin{equation}
q^{\left\lfloor \frac{n-6}{4} \right\rfloor}c_q\left(n-4\left\lfloor \frac{n-6}{4} \right\rfloor\right) - c_q(n)\nonumber 
\end{equation}
where $$c_q(n)=\frac{n-3}{6} \left\lfloor \frac{(n-2)(n-4)}{4} \right\rfloor+\frac{1}{q-1}\left\lfloor \frac{n^2-2n+2}{2} \right\rfloor+\frac{4n+4}{(q-1)^2}+\frac{16}{(q-1)^3}.$$

\end{lemma}
\begin{proof} Let $m=k-2-\val(1+u)$, $a=\val(x)+k$, and $b=\val(y)$. The inequalities become 
$m+3\leq  k\leq  a-m-2$, $a\leq  b \leq n-a$, $2m+5\leq  a \leq \left\lfloor \frac{n}{2}\right\rfloor$, 
$0\leq  m\leq \left\lfloor \frac{n-10}{4}\right\rfloor=:N$. Let $B=\frac{n-7}{6}   \left\lfloor \frac{(n-6)(n-8)}{4} \right\rfloor$. The number of double cosets is $q-1$ times \begin{eqnarray}
 &  & \sum\limits_{m=0}^{N}q^m \sum\limits_{a=2m+5}^{\left\lfloor\frac{n}{2}\right\rfloor}(n-2a+1)(a-2m-4)\nonumber \\
 & =&  \sum\limits_{m=0}^{N}q^m \sum\limits_{a=1}^{\left\lfloor\frac{n-8-4m}{2}\right\rfloor}a(n-5-4m-2(a+1))\nonumber \\
 & =& \sum\limits_{m=0}^{N}q^m \left(\frac{n-7-4m}{6}\right)\left\lfloor\frac{(n-8-4m)(n-6-4m)}{4}\right\rfloor \nonumber \\
&=& \sum\limits_{m=0}^{N} q^m\left( B -  \left\lfloor \frac{n^2-18n+82}{2} \right\rfloor m+ (2n-22)m(m-1) -\frac{8m(m-1)(m-2)}{3} \right)\nonumber \\
&=& \left( B -  \left\lfloor \frac{n^2-18n+82}{2} \right\rfloor q\frac{d}{dq}+ (2n-22)q^2\frac{d^2}{dq^2}  -\frac{8}{3}q^3\frac{d^3}{dq^3} \right)\left( \frac{q^{\left\lfloor \frac{n-6}{10} \right\rfloor}-1}{q-1}  \right) .\nonumber
\end{eqnarray}  The rest of the computation is tedious but elementary. 
\end{proof}


\subsection{Character table computations}

In this section, we carry out computations using the character tables of $\GSp(4, \Fq)$. These can be found in  \cite{En1972} when $q$ is even and in \cite{Sh1982} when $q$ is odd. We will freely make use of the notations in those papers for the conjugacy classes and (cuspidal) characters.

\begin{lemma}\label{nongeneric character lemma}
Suppose $\sigma$ is a nongeneric cuspidal irreducible representation of $\GSp(4, \Fq)$ and $R = \left\{\begin{bsmallmatrix}
1 & & & x\\
& 1 & & \\
& & 1 & \\
& & & 1
\end{bsmallmatrix} : x\in \Fq\right\}$. Then $\sigma^R=0$.
\end{lemma}
\begin{proof}
This follows immediately from the character table of $\GSp(4,\Fq)$. 
\end{proof}

For the remainder of this subsection, $\sigma$ will denote a generic cuspidal irreducible representation of $\GSp(4, \Fq)$ with trivial central character. In the notation of \cite{En1972} and \cite{Sh1982}, these are the ones of the form $\chi_5(k)$, $\chi_4(k, l)$, $X_4(\Theta)$, and $X_5(\Lambda, \omega)$.

\begin{lemma} \label{Klingen Levi}
Let $M$ and $R$ denote the following subgroups of $\GSp(4, \Fq)$: 
\begin{eqnarray}
M &=& \left\{\begin{bsmallmatrix}
1 &  & \\
& D &  \\
 &   &   \det(D)
\end{bsmallmatrix}: D\in \GL(2, \Fq) \right\}Z(\GSp(4, \Fq)) \nonumber \\
R &=& \left\{\begin{bsmallmatrix}
1 &  & \\
& D &  \\
x &   &   1
\end{bsmallmatrix}: D\in \SL(2, \Fq), \ x\in \Fq \right\}.\nonumber 
\end{eqnarray}

 Then \begin{eqnarray}
\sigma^R &=& 0 \nonumber \\
\dim \sigma^M &=& \begin{cases}
   0   & \text{ if }  \sigma = \chi_5(k) \text{ or } \sigma =X_4(\Theta)\\
  2 &  \text{ if } \sigma = \chi_4(k, l)  \text{ or } \sigma = X_5(\Lambda, \omega)
\end{cases}.\nonumber
\end{eqnarray}

\end{lemma}

\begin{proof} Let $M_1 = M\cap R$. We will show that if $\sigma = \chi_5(k)$ or $\sigma = X_4(\Theta)$, then $\sigma^{M_1}=0$. Suppose first that $q$ is even and $\sigma= \chi_5(k)$. We have $|M_1\cap A_1|=1$ and $|M_1\cap A_2|=q^2-1$. The character table then shows that \begin{eqnarray}
\dim\sigma^{M_1} & = & \frac{1}{|M_1|} \sum\limits_{m\in M_1}\trace\sigma(m) \nonumber \\
& = & \frac{1}{(q^2-1)q}\left(  \sum\limits_{m\in M_1\cap A_1}\trace\sigma(m) + \sum\limits_{m\in M_1\cap A_2}\trace\sigma(m)    \right) \nonumber \\
& = &  \frac{1}{(q^2-1)q} \left(  ( q^2-1)^2 + (q^2-1)(-q^2+1)  \right)\nonumber \\
& = & 0.\nonumber 
\end{eqnarray}
If $q$ is odd and $\sigma = X_4(\Theta)$ then the computation is identical to the previous case 
except the notation for the conjugacy classes has $A_1$ and $A_2$ replaced by $A_0$ and $A_1$, respectively. 


Suppose now that $q$ is even and $\sigma = \chi_4(k, l)$. The set $R\cap A_2$ consists of the $q^2-1$ elements where $x=0$ and $D$ is unipotent together with the $q-1$ elements with $x\neq 0$ and $D=I_2$. The set $R\cap A_{32}$ consists of the $(q-1)(q^2-1)$  matrices 
where $x\neq 0$ and $D$ is unipotent. So $|R\cap A_2| = q^2+q-2$ and $|R\cap A_{32}| = (q-1)(q^2-1)$. For each $1\leq i\leq q/2$, we have $|R\cap D_3(i)| = (q-1)|R\cap C_3(i)|$, since $R\cap D_3(i)$ and $R\cap C_3(i)$ consists, respectively, of matrices with $x\neq 0$ and $x=0$ and with $D$ having fixed irreducible characteristic polynomial with constant term 1. The character table then gives $\sum\limits_{r\in R\cap (C_3(i)\cup D_3(i))}\trace\sigma(r)=0$ and 
\begin{eqnarray}
\dim \sigma^R & = & \frac{1}{|R|}\sum\limits_{r\in R} \trace \sigma(r) \nonumber \\
& = & \frac{1}{|R|}\left( \sum\limits_{r\in R\cap A_1} \trace \sigma(r) +\sum\limits_{r\in R\cap A_2} \trace \sigma(r)  + \sum\limits_{r\in R\cap A_{32}} \trace \sigma(r)          \right)\nonumber \\
& = & \frac{1}{|R|} \left( (q-1)^2(q^2+1) + (q^2+q-2)(q-1)^2 + (q-1)(q^2-1)(1-2q)  \right) \nonumber \\
&=& 0. \nonumber
\end{eqnarray} 

Now suppose $q$ is odd and $\sigma = X_5(\Lambda, \omega)$. The set $R\cap A_1$ consists of the $q^2-1$ elements where $x=0$ and $D$ is unipotent together with the $q-1$ elements with $x\neq 0$ and $D=I_2$. The set $R\cap A_{21}$ consists of the $(q-1)(q^2-1)/2$  matrices where $-x\in \Fq^\times$ is a square and $D$ is unipotent. The set $R\cap A_{22}$ consists of the $(q-1)(q^2-1)/2$  matrices where $-x$ is a nonsquare and $D$ is unipotent.  So $|R\cap A_{1}| = q^2+q-2$ and $|R\cap A_{21}| =|R\cap A_{22}| = (q-1)(q^2-1)/2$. The set $R\cap B_0$ is the matrix with $x=0$ and $D=-I_2$. The set $R\cap B_1$ consists of matrices with $x\neq 0$ and $D=-I_2$, so $|R\cap B_1|=q-1$. The set $R\cap B_2$ consists of matrices with $x= 0$ and $-D$ unipotent so $|R\cap B_2|=(q-1)(q^2-1)$. The set $R\cap B_{31}$ consists of matrices with $x\neq 0$ a square  and $-D$ unipotent, so $|R\cap B_{31}|=(q-1)(q^2-1)/2$. The set $R\cap B_{32}$ consists of matrices with $x$ a nonsquare  and $-D$ unipotent, so $|R\cap B_{32}|=(q-1)(q^2-1)/2$. For each $u\in \mathbb{F}_{q^2}^\times$ with $u^{q+1}=1$, $u\neq \pm1$, the pair $(u, u^{-1})$ determines conjugacy classes $G_0$ and $G_1$. We have $|R\cap G_1| = (q-1)|R\cap G_0|$, since $R\cap G_1$ and $R\cap G_0$ consists, respectively, of matrices with $x\neq 0$ and $x=0$ and with $D$ having eigenvalues $u^{\pm1}$. The character table then gives $\sum\limits_{r\in R\cap (B_0\cup B_1)}\trace\sigma(r)=\sum\limits_{r\in R\cap (B_2\cup B_{31}\cup B_{32})}\trace\sigma(r)=\sum\limits_{r\in R\cap (G_0\cup G_1)}\trace\sigma(r)=0$ and \begin{eqnarray}
\dim \sigma^R & = & \frac{1}{|R|}\sum\limits_{r\in R} \trace \sigma(r) \nonumber \\
& = & \frac{1}{|R|}\left( \sum\limits_{r\in R\cap A_0} \trace \sigma(r) +\sum\limits_{r\in R\cap A_1} \trace \sigma(r)  + \sum\limits_{r\in R\cap A_{21}} \trace \sigma(r)     + \sum\limits_{r\in R\cap A_{22}} \trace \sigma(r)       \right)\nonumber \\
& = & \frac{1}{|R|} \left( (q-1)^2(q^2+1) + (q^2+q-2)(q-1)^2 + \frac{1}{2}(q-1)(q^2-1)(1-q+1-3q)  \right) \nonumber \\
&=& 0. \nonumber
\end{eqnarray}

Next, since $M = Z(\GSp(4, \Fq))\times M'$ where $$M' = \left\{\begin{bmatrix}
1 & & \\
& D & \\
& & \det(D)
\end{bmatrix}  : D\in \GL(2, \Fq)\right\},$$ we have $$\dim\sigma^M = \frac{1}{|M|}\sum\limits_{m\in M} \trace \sigma(m)=\frac{1}{|M'|}\sum\limits_{m\in M'} \trace \sigma(m).$$ 

Suppose first that $q$ is even and $\sigma = \chi_4(k, l)$ . We have $|M'\cap A_1|=1$ and $|M'\cap A_2|=q^2-1$. 
The only additional contributions are from conjugacy classes $C_3(i)$, with $1\leq i\leq q/2$, in which $D$ is elliptic and $\det(D)=1$. From the character table, and the fact that there are $q^2-q$ elliptic elements in $\SL(2, \Fq)$ with a given eigenvalue pair, we have 
\begin{eqnarray}
\sum\limits_{i=1}^{q/2}\sum\limits_{m\in M'\cap C_3(i)} \trace\sigma(m) &=& -(q^2-q)(q-1) \sum\limits_{1\leq i\leq q/2} (\zeta_{q+1}^{ik}+\zeta_{q+1}^{-ik} + \zeta_{q+1}^{il} + \zeta_{q+1}^{-il}) \nonumber\\
&=& 2(q^2-q)(q-1) \nonumber
\end{eqnarray} and therefore\begin{eqnarray}
\dim \sigma^M &=& \frac{1}{(q+1)q(q-1)^2} \left( (q-1)^2(q^2+1) + (q^2-1)(q-1)^2 + 2(q^2-q)(q-1)  \right) \nonumber \\
&=& 2. \nonumber
\end{eqnarray}

If $q$ is odd and $\sigma = X_5(\Lambda, \omega)$, then the additional relevant families of conjugacy classes are $B_0, B_2$, and $G_0$.  We have $|M'\cap B_0| =1$ and $|M'\cap B_2| = q^2-1$. The elements of $M'$ in a class $G_0$ have $D\in \SL(2, \Fq)$ elliptic. If the eigenvalues of $D$ are $u^{\pm 1}$, then $\trace\sigma(m) = (1-q)(\omega(u) + \omega(u^{-1}) + \Lambda(u) \omega(u^{-1}) + \Lambda(u^{-1} ) \omega(u) ) $. Since $$\sum\limits_{u: \ u^{q+1}=1}\left(\omega(u) + \omega(u^{-1}) + \Lambda(u) \omega(u^{-1}) + \Lambda(u^{-1} ) \omega(u)  \right)=0$$ we have 
$$\sum\limits_{u\neq \pm1,\  u^{q+1}=1} \left(\omega(u) + \omega(u^{-1}) + \Lambda(u) \omega(u^{-1}) + \Lambda(u^{-1} ) \omega(u)  \right) \nonumber  = -4(1+\omega(-1)).$$ Therefore the total contribution from these elements to the sum, via the character table is 
$-4(1+\omega(-1))\cdot \frac{-1}{2}(q-1) (q^2-q) =2q(q-1)^2(1+\omega(-1)).$ Therefore \begin{eqnarray}
 \dim \sigma^M &=& \frac{(q-1)^2\left( (q^2+1)+ (q^2-1) + 2\omega(-1) - 2\omega(-1)(q+1) + 2q(1+\omega(-1))\right)  }{(q^2-1)(q^2-q)}  \nonumber \\
 &=& 2.\nonumber
\end{eqnarray}

This concludes the proof. \end{proof}

\begin{lemma}
If $S = \left\{  \begin{bsmallmatrix}
a  & & & \\
& b & & \\
* & & a & \\
* & * & & b 
\end{bsmallmatrix}\in \GSp(4, \Fq) \right\} $ then $\dim \sigma^S = q-1$.
\end{lemma}

\begin{proof}
We will compute $\dim \sigma^S = \frac{1}{|S|} \sum\limits_{s\in S}\trace\sigma(s)  $. 

Suppose first that $q$ is even. Since $\GSp(4, \Fq) = \Fq^\times \times \Sp(4, \Fq)$, we may replace $S$ by $S':=S\cap \Sp(4, \Fq)$ and assume $a=b^{-1}$. For both $\sigma = \chi_5(k)$ and $\sigma = \chi_4(k, l)$, the only contributions come from the unipotent elements in $S$ and the character table shows \begin{eqnarray}
\dim \chi_5(k)^S &=& \frac{1}{q^2(q-1)} \sum\limits_{x, z\in \Fq} \trace \sigma\begin{bsmallmatrix}
1 & & & \\ 
& 1 & & \\
x & & 1 & \\
z & x & & 1
\end{bsmallmatrix} \nonumber \\
& = & \frac{(q^2-1)^2-(q-1)(q^2-1)-(q-1)(q^2-1)+(q-1)^2 }{q^2(q-1)} \nonumber \\
& = & q-1\nonumber\\
\dim \chi_4(k,l)^S &=& \frac{1}{q^2(q-1)} \sum\limits_{x, z\in \Fq} \trace \sigma\begin{bsmallmatrix}
1 & & & \\ 
& 1 & & \\
x & & 1 & \\
z & x & & 1
\end{bsmallmatrix} \nonumber \\
& = & \frac{(q^2+1)(q-1)^2+(q-1)(q-1)^2+(q-1)(q-1)^2-(q-1)^2(2q-1) }{q^2(q-1)} \nonumber \\
& = & q-1\nonumber
\end{eqnarray} where the second equalities come from the fact that the conjugacy class of the matrix is $A_1$ iff $x=z=0$, is $A_2$ iff $z\neq x=0$, is $A_{31}$ iff $x\neq z=0$, and is $A_{32}$ iff $xz\neq 0$. 

Now suppose that $q$ is odd. For both $\sigma = X_4(\Theta)$ and $\sigma = X_5(\Lambda, \omega)$, the only contributions come from the scalar multiples of unipotent elements in $S$. In this case, one has the conjugacy class $A_0$ iff $x=z=0$, $A_1$ iff $z\neq x=0$, and $A_{21}$ iff $x\neq 0$. The character table shows \begin{eqnarray}
\dim X_4(\Theta)^S &=& \frac{1}{q^2(q-1)} \sum\limits_{x, z\in \Fq} \trace \sigma\begin{bsmallmatrix}
1 & & & \\ 
& 1 & & \\
x & & 1 & \\
z & x & & 1
\end{bsmallmatrix} \nonumber \\
& = & \frac{(q^2-1)^2-(q-1)(q^2-1)-q(q-1)(q-1)}{q^2(q-1)} \nonumber \\
& = & q-1\nonumber\\
\dim X_5(\Lambda, \omega)^S &=& \frac{1}{q^2(q-1)} \sum\limits_{x, z\in \Fq} \trace \sigma\begin{bsmallmatrix}
1 & & & \\ 
& 1 & & \\
x & & 1 & \\
z & x & & 1
\end{bsmallmatrix} \nonumber \\
& = & \frac{(q^2+1)(q-1)^2+(q-1)(q-1)^2-q(q-1)(q-1) }{q^2(q-1)} \nonumber \\
& = & q-1.\nonumber
\end{eqnarray} This concludes the proof.
\end{proof}

\begin{lemma}
Define the following subgroups of $\GSp(4, \Fq)$: \begin{eqnarray}
A&=& \left\{  \begin{bsmallmatrix}
a & & & \\
x & ad & & \\
& & a/d & \\
z & & -x/d & a
\end{bsmallmatrix}: a, d\in \Fq^\times, x, z\in \Fq\right\} \nonumber \\
B &=& \left\{\begin{bsmallmatrix}
a & & & \\
& b & & \\
& x & c/b & \\
& & & c/a
\end{bsmallmatrix} : a, b, c\in \Fq^\times, x\in \Fq\right\}\nonumber \\
C &=& \left\{\begin{bsmallmatrix}
a & & & \\
& b & & \\
& x & c/b & \\
y & & & c/a
\end{bsmallmatrix}:  a, b, c\in \Fq^\times, x, y\in \Fq\right\}\nonumber \\
D &=& \left\{\begin{bsmallmatrix}
a & a y & & \\
& b & & \\
& x & c/b & -cy/b \\
 & & & c/a
\end{bsmallmatrix}:  a, b, c\in \Fq^\times, x, y\in \Fq\right\}\nonumber 
\end{eqnarray}

Then $\dim \sigma^A = q-1$, $\dim \sigma^B = q+1$ and $\dim \sigma^C=\dim\sigma^D = 1$. 
\end{lemma}

\begin{proof}
The proof of this is very similar to the computations of the previous Lemma.
\end{proof}




\begin{lemma} \label{Last character computation}
Suppose $R= \left\{ \begin{bsmallmatrix}
1 & & & \\
x & 1 & & \\
y & & 1 & \\
z & y & -x & 1
\end{bsmallmatrix}  \begin{bsmallmatrix}
1 & & & \\
 & 1 & & \\
 & x & 1 & \\
 &  &  & 1
\end{bsmallmatrix} :  x, y , z\in \Fq \right\}.$ 
Then $\dim \sigma^R = q-1$.
\end{lemma}

\begin{proof} Note $|R|=q^3$. Suppose first that $q$ is odd. Observe that $R\cap A_0$ consists of the identity element and $R\cap A_1$ consisting of matrices of the form $\begin{bsmallmatrix}
1 &  &  &  \\
& 1 &  &  \\
& & 1 &  \\
c & & & 1
\end{bsmallmatrix}$ with $c\neq 0$. Meanwhile $R\cap A_3$ consists of the $q^2(q-1)$ matrices of the form $\begin{bsmallmatrix}
1 &  &  &  \\
x & 1 &  &  \\
*& x & 1 &  \\
*&* & -x & 1
\end{bsmallmatrix}$ with $x\neq 0$; these are precisely the matrices in $R$ with a 1-dimensional eigenspace. Finally $R\cap A_{21}$ consists of the $q(q-1)$ matrices $\begin{bsmallmatrix}
1 &  &  &  \\
& 1 &  &  \\
b& & 1 &  \\
c& b & & 1
\end{bsmallmatrix}$ with $b\neq 0$. 
In particular, $|R\cap A_{22}|=0$. From the character table we have \begin{eqnarray}
\dim X_4(\Theta)^R 
&=& q^{-3} \left( (q^2-1)^2 -(q-1)(q^2-1)  - q(q-1)^2 + q^2(q-1)  \right) \nonumber \\
&=& q-1\nonumber \\
\dim X_5(\Lambda, \omega)^R &=& q^{-3} \left( (q^2+1)(q-1)^2 + (q-1)^3 - q(q-1)^2 + q^2(q-1)  \right) \nonumber \\
&=& q-1.\nonumber
\end{eqnarray}

Now suppose $q$ is even. In this case, we have $|R\cap A_1|=1$, $|R\cap A_2| = q-1$ (the ones with a 3-dimensional eigenspace), and note that $|R\cap A_{41}\cup A_{32}| = q^2(q-1)$, again consisting of the matrices with a one-dimensional eigenspace; we do not need to distinguish between the two conjugacy classes because the relevant characters take the same value on both. If $b, c\neq 0$ then $\begin{bsmallmatrix}
1 & &  &  \\
& 1 & &  \\
b & & 1 & \\
c & b & & 1
\end{bsmallmatrix} \in A_{32}$ (we can conjugate by the diagonal matrix $(b, b, 1, 1)$) to reduce to $b=1$ and then by $(c^{1/2}, c^{-1/2}, c^{1/2}, c^{-1/2})$ to make $c=1$). Thus $|R\cap A_{31}| = q-1$ and $|R\cap A_{32}| = (q-1)^2$. The character table shows that 
\begin{eqnarray}
\dim \chi_5(k)^R &=& q^{-3}\left( (q^2-1)^2 - (q-1)(q^2-1) - (q-1)(q^2-1) + (q-1)^2 + q^2(q-1)  \right) \nonumber \\
&=& q-1 \nonumber\\
\dim \chi_4(k,l)^R &=& q^{-3}\left( (q^2+1)(q-1)^2 + (q-1)^3 + (q-1)^3 - (q-1)^2(2q-1) + q^2(q-1)  \right)  \nonumber \\
&=& q-1. \nonumber
\end{eqnarray} This concludes the proof.
\end{proof}

\section{Main result}

Recall the notations for characters of $\GSp(4, \Fq)$ from \cite{En1972} and \cite{Sh1982} mentioned in the previous section.

\begin{theorem}\label{Main Theorem} Let $\pi=\cInd_{ZK}^G(\sigma)$ be a generic depth zero supercuspidal representation of $\GSp(4, F)$ with unramified central character. Let $n\geq1$. 


a) If $\sigma = \chi_5(k)$ or $\sigma=X_4(\Theta)$ then

$$\dim\pi^{\Kl{n}}=\frac{  ((q-1)c_n(q)+72)q^{\left\lfloor \frac{n-2}{4} \right\rfloor}  -72}{(q-1)^2}-\frac{18(n+2)}{q-1}-n(n+3)$$

where $$c_n(q)= \begin{cases}
q^2+36q+71  & \text{ if }n\in 4\Z   \\
    4q^2+50q+72 & \text{ if }n\in 1+4\Z \\
 11q+61  & \text{ if }n\in 2+4\Z\\
   22q+68  & \text{ if }n\in 3+4\Z
\end{cases}.$$ 








b) If $\sigma = \chi_4(k, l)$ or $\sigma=X_5(\Lambda, \omega)$ then $\dim\pi^{\Kl{n}}$ equals the above expression plus $2\lfloor \frac{n-1}{2}\rfloor$. 

In particular, $\dim \pi^{\Kl{n}}$ is a polynomial in $q$ of degree $\left\lfloor \frac{n}{4}\right\rfloor$ for all $n\geq 0$ and all $\sigma$.

\end{theorem}

\begin{proof}
This now follows from the previous sections after some routine algebra. 
\end{proof}

\begin{corollary}
Let $n\geq 1$. If $q=2$ and $\sigma = \chi_5(k)$ then $$\dim\pi^{ \Kl{n}}=-(n+9)(n+12)+2^{\left\lfloor \frac{n-2}{4}\right\rfloor}   \begin{cases}
219  & \text{ if }n\in 4\Z   \\
    260 & \text{ if }n\in 1+4\Z \\
 155 & \text{ if }n\in 2+4\Z\\
  184 & \text{ if }n\in 3+4\Z
   
\end{cases}.$$ If $q=3$ and $\sigma =X_4(\Theta)$ then $$\dim\pi^{ \Kl{n}}=-(n+6)^2+3^{\left\lfloor \frac{n-2}{4}\right\rfloor}   \begin{cases}
112  & \text{ if }n\in 4\Z   \\
    147 & \text{ if }n\in 1+4\Z \\
 65 & \text{ if }n\in 2+4\Z\\
  85 & \text{ if }n\in 3+4\Z
   
\end{cases}.$$
\end{corollary}


\begin{remark}\label{L packet remark}
Using our character table computations, together with the explicit determination of $L$-packets done in \cite{L2010}, the condition $\dim\sigma^M=0$ (where $M$ is the Klingen Levi subgroup) is equivalent to $\sigma$ being $\chi_5(k)$ or $X_4(\Theta)$, and is also equivalent to the $L$-packet of $\pi$ being a singleton, and to the $L$-parameter of $\pi$ being irreducible as a 4-dimensional representation of the Weil-Deligne group. 
\end{remark}

 \begin{landscape}
\begin{table}[h!] 
  \begin{center}
    \caption{$\Kl{n}$ fixed vector information, $n\geq 1$}
    \label{tab:table1}
    \begin{tabular}{c|c|c|c|c|c|c|c} 
   $g$ &   $i$ & $j$  & $k$ &
      $R_g$ & $\dim \sigma^{R_g}$ &  $|\{ g \}|$\\
      \hline

       \   $t_{i,j}$ &  $[2-n, 0]$ & $[1-2i, n-i-1 ]$ & & $\begin{bsmallmatrix}
                  * & * & & \\  & * &  & \\ & *  & * &  *\\ & &  & * 
                  \end{bsmallmatrix}$ & 1 &  $\frac{n(n-1)}{2}$ \\
                  
  \   $t_{i,j}$ & $[1, n-2 ]$ & $[ \max(1,n-2i) , n-i-1  ]$ & & $\begin{bsmallmatrix}
     * & & & \\  & * &  & \\ & * & * & \\ *  & &  & * 
      \end{bsmallmatrix}$ & 1 & $\left\lfloor \frac{(n-1)^2}{4} \right\rfloor$ \\
      
   \    $t_{i,j}$ &  $[1, \lfloor \frac{n-1}{2} \rfloor ]$ & 0  & & $\begin{bsmallmatrix}
       * &  & & \\  & *  & *  & \\ & * & * &  \\ & &  & * 
       \end{bsmallmatrix}$ &  0 \text{ or } 2  & $\left\lfloor \frac{n-1}{2} \right\rfloor$ \\

   \    $t_{i,j}$ & $[1, \lfloor\frac{n-2}{2} \rfloor]$ & $[1, n-2i-1]$ & & $\begin{bsmallmatrix}
      * & & & \\  & * &  & \\ & * & * & \\  & &  & * 
      \end{bsmallmatrix}$ & $q+1$  & $\left\lfloor \frac{(n-2)^2}{4} \right\rfloor$  \\

  \   $t_{i,j}Z_{i+j+k}$ & $[2, n-3]$  &  $[1, n-i-2]$ &  $[1, \min(i, n-i-j )-1]$  & $\begin{bsmallmatrix}
                     a & & & \\  & ad &  & \\ & * & \frac{a}{d} & \\ *  & &  & a 
                     \end{bsmallmatrix}$  &  $q-1$  & $\left\lfloor \frac{(n-1)(n-3)(2n-7)}{24} \right\rfloor $   \\
                     
  \  $t_{i,j}X_k$  & $[2, n-3]$ & $[1, n-i-2]$ &  $[1, \min(i, n-i-j)-1]$       & $\begin{bsmallmatrix}
                 a & & & \\ *  & a &  & \\   & & d & \\ * &  & *  & d 
                 \end{bsmallmatrix}$ & $q-1$   & $\left\lfloor \frac{(n-1)(n-3)(2n-7)}{24} \right\rfloor$ \\
       
 \   $t_{i,j}Y_k$ & $[1,   \lfloor\frac{n-4}{2} \rfloor]$ & $[3, n-2i-1]$  &  $[i+\lceil \frac{j+1}{2} \rceil , i+j-1]$   & $\begin{bsmallmatrix}
                    a & & & \\  & d  &  & \\ *  & *   & a  & \\ & *  &  & d 
                    \end{bsmallmatrix}$  & $q-1$  & $\frac{n-3}{6} \left\lfloor \frac{n^2-6n+8}{4} \right\rfloor$   \\

    \   $S(\vec{k})$ & - & -  &  -   & $\begin{bsmallmatrix}
                        a & & & \\ v & a  &  & \\ *  & v   & a  & \\ * & *  &  -v & a 
                        \end{bsmallmatrix}$  & $q-1$  & \ref{number with z<xy}, \ref{number with z=xy and 1+z/xy a unit}, \ref{number with z=xy and 1+z/xy not a unit}

    \end{tabular}
  \end{center}
\end{table}

\end{landscape}

\end{document}